\documentclass[a4paper, 12pt,reqno]{amsart}
\usepackage[12pt]{extsizes}
\usepackage{amsmath,amssymb}
\usepackage{amscd}
\usepackage{geometry}
\usepackage{amsthm}
\usepackage{euscript}
\usepackage{latexsym}
\usepackage[matrix,arrow,curve,cmtip]{xy}
\usepackage[cp1251]{inputenc}
\usepackage[english]{babel}
\usepackage{mathrsfs}
\usepackage{graphicx}
\usepackage{keyval}

\usepackage{mathtools}
\usepackage{verbatim}
\usepackage{hyperref}
\usepackage{pb-diagram}
\usepackage[symbol*]{footmisc}

\newtheorem{theorem}{Theorem}[section]
\newtheorem{lemma}{Lemma}[section]
\newtheorem{proposition}{Proposition}[section]
\newtheorem{corollary}{Corollary}[section]

\theoremstyle{definition}
\newtheorem{definition}[theorem]{Definition}
\newtheorem{example}[theorem]{Example}
\newtheorem{remark}[theorem]{Remark}

\numberwithin{equation}{section}

\newcommand{\D}{\mathscr{D}}

\newcommand{\m}{\mathfrak}

\begin{document}

\pagestyle{plain}
\title{\bf DERIVATIONS OF LEAVITT PATH ALGEBRA}

\maketitle

\begin{center}
Viktor Lopatkin\footnote{South China Normal University, Guangzhou, China, \texttt{wickktor@gmail.com}}
\end{center}
\begin{abstract}
  In this paper, we describe the $K$-module $HH^1(L_K(\Gamma))$ of outer derivations of the Leavitt path algebra $L_K(\Gamma)$ of a row-finite graph $\Gamma$ with coefficients in an associative commutative ring $K$ with unit. We explicitly describe a set of generators of $HH^1(L_K(\Gamma))$ and relations among them. We also describe a Lie algebra structure of outer derivation algebra of the Toeplitz algebra. We prove that every derivation of a Leavitt path algebra can be extended to a derivation of the corresponding $C^*$-algebra.

\medskip


\end{abstract}

\section*{Introduction}
Throughout this note, all rings are assumed to be nonzero, associative with unit.

Given a row-finite directed graph $\Gamma$ and a field $K$, G. Abrams and A. Pino in \cite{AbP}, and independently P. Ara, M.A. Moreno, E. Pardo in \cite{AMP}, introduced the \textit{Leavitt path algebra $L_K(\Gamma)$}. Later, in \cite{T11} M. Tomforde generalized the construction of Leavitt path algebras by replacing the field with a commutative ring. This algebra is an algebraic analog of graph Cuntz--Krieger $C^*$-algebras. These algebras have attracted significant interest and attention, not only from ring theorists, but from analysts working in $C^*$-algebras, group theorists, and symbolic dynamicists as well (see for example \cite{30, AraGood, 73} and the survey \cite{Abrams-surv}). In \cite{AraGood} it has been proved that the Leavitt path algebra and their generalizations are hereditary algebras. It follows that the higher degree homologies vanish. P. Ara and G. Corti\~nas \cite{AraCortinas} calculated the Hochschild homology of Leavitt path algebras. But this result does not yield explicit formulas for the generators of $HH_*(L_K(\Gamma))$.

In this paper we deal with the Leavitt path algebra $L_K(\Gamma)$ of a row-finite (no necessary finite) graph $\Gamma$ with coefficients in a commutative ring $K$. We describe every derivation of the Leavitt path algebra $L_K(\Gamma)$ via explicit formulas (see Theorem \ref{maintheorem}). This description allows to describe the $K$-module $HH^1(L_K(\Gamma))$ of outer derivation (=the first Hochschild cohomology) of the Leavitt path algebra $L_K(\Gamma)$ (see Theorem \ref{genthm}) and  prove that every derivation of $L_\mathbb{C}(\Gamma)$ can be extended to the derivation of the $C^*$-algebra $C^*(\Gamma)$ (see Theorem \ref{Extend}). We also show that the vector space $HH^1(L_K(\Gamma))$, over a field $K$, has infinite dimension if the graph $\Gamma$ contains a path with infinite length and dimension $0$ otherwise. Finally, we describe a Lie algebra structure of outer derivation algebra of the Toeplitz algebra (see Section 5).

\vspace{2ex}

{\bf Acknowledgements:}  the author would like to express his deepest gratitude to {\sc Prof. Leonid A. Bokut'}, who has drawn the author's attention to the problems studied in this work. Special thanks are due to {\sc Prof. Efim Zelmanov} for very useful discussions and for having kindly clarified some very important details. I am also extremely indebted to my friend, my Chinese brother {\sc Dr. Zhang Junhuai} for the great support, without which the author's life would be very difficult. I am deeply grateful to {\sc Dr. Pasha Zusmanovich}, who suggested corrections and improvements.

\section{Leavitt path algebra}

We present the main definitions and notations which will use in this note. We also give a list of examples of known algebras which arise as Leavitt path algebras. Then we give a description of a basis ($K$-basis) of the Leavitt path algebra of a row-finite graph with coefficient in a commutative ring, which is an analog of Zelmanov et al'.s result (see \cite[Theorem 1]{Zel}).

\subsection{Definitions and Examples}

We begin by recalling some general notions of graph theory: a directed graph $\Gamma = (V,E, s, r)$ consists of two sets $V$ and $E$, called the set of vertices and edges, respectively, and two maps $s, r:E \to V$ called {\it source} and {\it range} (of edge), respectively.  For conveniences, we also set $s(v) = r(v) = v$ for every vertex $v\in V$. The graph is called {\it row-finite} if for all vertices $v \in V$, $|s^{-1}(v)| < \infty$. A vertex $v$ for which $s^{-1}(v)$ is empty is called a {\it sink}. A {\it path} $\m{p} = e_1\cdots e_\ell$ in the graph $\Gamma$ is a sequence of edges $e_1, \ldots, e_\ell \in E$ such that $r(e_i) = s(e_{i+1})$ for $i = 1, \ldots, \ell -1$. In this case we say that the path $\m{p}$ {\it starts} at the vertex $s(e_1)$ and {\it ends} at the vertex $r(e_\ell)$, we put $s(\m{p}): = s(e_1)$ and $r(\m{p}):=r(e_\ell)$.

Let $\m{p} = p_0p_1\cdots p_{z-1}p_z$ be the decomposition of the path $\m{p}$ via the edges $p_0,\ldots, p_z\in E$; we use the following notation: we set $\m{p}/p_0: = p_1\cdots p_{z-1}p_z$ and $\m{p}\backslash p_z = p_0p_1\cdots p_{z-1}$. If $\m{p}=p_0$ (resp. $\m{p}= p_z$), then we set $\m{p}/p_0:=r(\m{p})$ (resp. $\m{p}\backslash p_z: = s(\m{p})$). Finally, if $\m{p} \in E$, then, just for conveniences, we set $[\m{p}/p_0]^*:=r(\m{p})$ and $[\m{p}\backslash p_z]^*:=s(\m{p})$.

\begin{definition}[{\textit{cf.}~\cite[Definition 1.3]{AbP} and \cite[Definition 2.4]{T11}}] Let $K$ be an commutative ring and $\Gamma = (V,E)$ be a row-finite graph. The Leavitt path algebra $L_K(\Gamma)$ of the graph $\Gamma$ {\it with coefficients in} $K$ is a $K$-algebra presented by the set of generators $\{v, v \in V\}$, $\{e,e^*:e \in E\}$ and the set of relations:
\begin{itemize}
\item[{\rm (1)}] $vu = \delta_{v,u}v$, for all $v,u \in V$;
\item[{\rm (2)}] $s(e)e = er(e) = e,$ $r(e)e^* = e^* s(e) = e^*,$ for all $e \in E;$
\item[{\rm (3)}] $e^*f = \delta_{e,f}r(e),$ for all $e,f \in E;$
\item[{\rm (4)}] $v = \sum_{s(e)=v}ee^*$, for an arbitrary vertex $v \in V \setminus\{\mathrm{sinks}\}.$
\end{itemize}
\end{definition}

\begin{example}[{\bf Full matrix algebras}] Let $A_n$ denote the following graph $ \xymatrix{
    \bullet^{v_1} \ar@{->}[r]^{e_1} & \bullet^{v_2} \ar@{->}[r]^{e_2} & \ar@{.>}[r] & \bullet^{v_{n-1}} \ar@{->}[r]^{e_{n-1}} & \bullet^{v_n}.
    }
$ Then $L_K(A_n) \cong \mathbb{M}_n(K)$, the full $n\times n$ matrix algebra. Let us describe this isomorphism. Let $\mathsf{E}_{i,j}$ be the standard matrix units, here $1 \le i,j\le j$, then the isomorphism $L_K(A_n) \cong \mathbb{M}_n(K)$ can be described as follows:
\begin{align*}
  &v_i \leftrightarrow \mathsf{E}_{i,i}, & 1\le i\le n,\\
  &e_i \leftrightarrow \mathsf{E}_{i,i+1}, &1 \le i \le n-1,\\
  &e_i^* \leftrightarrow \mathsf{E}_{i+1,i}, & 1 \le i \le n-1.
\end{align*}
\end{example}

\begin{example}[{\bf The Laurent polynomial algebra}]\label{Loran}
Let $R_1$ denote the graph $\xymatrix{
    \bullet^v \ar@(dl,ul)[]^e}.$ Then $L_K(R_1) \cong K[t,t^{-1}]$, the Laurent polynomial algebra. The isomorphism is clear: $v \leftrightarrow 1,$ $e \leftrightarrow t,$ $e^* \leftrightarrow t^{-1}.$
\end{example}

\begin{example}[{\bf Leavitt algebra}]
For $1 \le \ell \le \infty$, let $R_\ell$ be a graph with $\ell$ edges and with one vertex $v$. Then $L_K(R_\ell) \cong L_K(1,\ell)$, the Leavitt algebra of order $\ell$, defined by the generators $\{x_i,y_i: 1 \le i \le \ell\}$, and relations $y_ix_i = \delta_{i,j},$ $\sum_{i=1}^\ell x_iy_i =1.$
\end{example}

\begin{example}[{\bf The Toeplitz algebra}]\label{Toeplitz}
For any field $K$, the Jacobson algebra, described in \cite{J}, is a $K$-algebra: $A = K \langle x,y: xy =1 \rangle.$ This algebra was the first example appearing in the literature of an algebra which is not directly finite, that is, in which there are elements $x,y$ for which $xy =1$ but $yx \ne 1$. Let $\mathrm{T}$ denote the ``Toeplitz graph'' $\xymatrix{
    \bullet^v \ar@(dl,ul)^e \ar@{->}[r]^f& \bullet^u.
   }
$ Then $L_K(\mathcal{T}) \cong A$. The isomorphism is described as follows:
\begin{align*}
  &v \leftrightarrow yx,\qquad u \leftrightarrow 1-yx,\\
  &e \leftrightarrow y^2x,\qquad f \leftrightarrow y-y^2x,\\
  &e^* \leftrightarrow yx^2,\qquad f^*\leftrightarrow x - yx^2.
\end{align*}
\end{example}

\subsection{A basis of Leavitt path algebras over a commutative ring}
The main goal of this subsection is to give a basis of the Leavitt path algebra $L_K(\Gamma)$ of a row-finite graph $\Gamma$ with coefficients in a commutative ring $K$, which is an analog of Zelmanov et al'.s result \cite[Theorem 1]{Zel}.

For an arbitrary vertex $v \in V$ which is not a sink, let us consider the set of edges $\{e_1,\ldots, e_\ell\}$ with the common source $v = s(e_1) = s(e_2) = \ldots =s(e_\ell)$. Let us choose the edge $e_1\in E$. We will refer to this edge as {\it special}. In other words, we fix a function $\vartheta:V \setminus \{\mbox{sinks}\} \to E$ such that $s(\vartheta(v)) = v$ for an arbitrary $v \in V \setminus \{\mbox{sinks}\}$. Denote by $\vartheta(V)$ the set of all special edges of $\Gamma$. Just for convinces, we write $\delta_{e,\vartheta} = 1$ if the edge $e$ is special and $\delta_{e,\vartheta} = 0$ otherwise.

Let us consider the following set of polynomials of $L_K(\Gamma)$,
\begin{equation}\label{GSB}
\mathbf{GS} = \mathbf{GS}_1 \cup \mathbf{GS}_2 \cup \mathbf{GS}_3\cup\mathbf{GS}_4,
\end{equation}
where
\begin{eqnarray*}
  \mathbf{GS}_1 &:=& \bigcup_{\substack{v,u\in V\\e\in E}} \{vu - \delta_{v,u}v,\, ve - \delta_{v, s(e)}e,\,ev - \delta_{v,r(e)}e \},\\
  \mathbf{GS}_2 &:=& \bigcup_{\substack{v\in V\\e,f \in E}} \{ve^* - \delta_{v, s(e)}e^*,\,e^*v - \delta_{v,s(e)}e^*,\, e^*f - \delta_{e,f}r(e) \},\\
  \mathbf{GS}_3 &:=& \bigcup_{e,f\in E} \{ee^* -s(e) + \sum_{s(f)=s(e)} ff^*: \vartheta(s(e))\ne f \},\\
  \mathbf{GS}_4 &:=& \bigcup_{e,f\in E} \{ef:r(e) \ne s(f) \}\cup  \{e^*f^*: r(f) \ne s(e)  \}\\ && \cup  \{ef^*: r(e) \ne r(f)  \}.
\end{eqnarray*}

It can be shown by straightforward computations that all these polynomials are closed with respect to compositions (see \cite{B,Be} and the survey \cite{BokSurv}), \textit{i.e.,} $\mathbf{GS}$ is the Gr\"obner--Shirshov basis of $L_K(\Gamma)$. Thus, by Composition--Diamond lemma, in the case when $K$ being a field, the set of irreducible words, \textit{i.e}., not containing the leading monomials of polynomials of $\mathbf{GS}$, is a $K$-basis of $L_K(\Gamma)$ and we get Zelmanov et al'.s result \cite[Theorem 1]{Zel}. However, this result has the same form in the case when $K$ is a commutative ring (see \cite[Theorem 3.7 ]{L-N}).

\begin{theorem}[{\textit{cf.} \cite[Theorem 1]{Zel} and \cite[Theorem 3.7 ]{L-N}}]\label{GSBLev}
Let $\Gamma$ be a row-finite graph and $K$ a commutative ring. Then, the following elements form a $K$-basis $\m{B}$ of the Leavitt path algebra $L(\Gamma)$: (1) the set of all vertices $V$, (2) the set of all paths $\mathfrak{P}$, (3) the set $\mathfrak{P}^* := \{\m{p}^*: \m{p} \in \mathfrak{P}\},$ (4) a set $\mathfrak{M}$ of words of the form $\m{wh}^*$, where $\m{w} = w_0\cdots w_z \in \mathfrak{P}$, $\m{h}^* = [h_0 \cdots h_z]^* = h_z^* \cdots h_0^* \in \mathfrak{P}^*$, $w_i,h_j \in E$, are paths that end at the same vertex $r(\m{w}) = r(\m{h})$, with the condition that the edges $w_z$ and $h_z$ are either distinct or equal, but not special.
\end{theorem}

\begin{remark}
Just for convinces, we frequently use the following notations: (1) $\m{wh^*} \in \m{P}$ means $\m{h} = r(\m{w})$, (2) $\m{wh^*} \in \m{P}^*$ means $\m{w} = r(\m{h})$, and (3) $\m{wh^*} \in V$ means $\m{w} = \m{h} = v \in V$.
\end{remark}

\section{The Module of Derivations of Leavitt Path Algebra.}
We begin by recalling some general notions. Let $\Lambda$ be an algebra over a commutative ring $K$. Let us denote by $\mathbf{Der}_K(\Lambda)$ (or shortly $\mathbf{Der}(\Lambda)$) the $K$-module of all derivations of $\Lambda$. As well known, every element $\lambda\in\Lambda$ determines \textit{the inner derivation} $\mathbf{ad}_\lambda$, $\mathbf{ad}_{\lambda}(a) = \lambda a - a \lambda$, $a\in \Lambda$. Denote by $\mathbf{ad}(\Lambda)$ the $K$-module of inner derivations of $\Lambda$. Further, the $K$-module $HH^1(\Lambda): = \mathbf{Der}(\Lambda)/ \ \mathbf{ad}(\Lambda)$ is called \textit{the first Hochschild cohomology of} $\Lambda$ or \textit{the module of outer derivations of $\Lambda$.}

Let $\Gamma$ be a row-finite graph, $K$ a commutative ring, $L_K(\Gamma)$ the Leavitt path algebra over $K$ and $\mathscr{D}$ a derivation of $L_K(\Gamma)$. In this section, we aim to describe a set of generators of the $K$-module $\mathbf{Der}(L_K(\Gamma))$.

For $L_K(\Gamma)$ we have the $K$-basis $\m{B}$ (see Theorem \ref{GSBLev}) and we can write $\mathscr{D}(x) = \sum_{\m{b} \in \m{B}}{F}_{\m{b}}(x)\m{b}$, for every $x \in L_K(\Gamma)$ where $F_{\m{b}}:L_K(\Gamma) \to K$ are functionals and almost all $F_{\m{b}}(x)$ are zero. Conversely, let $\{F_\m{b}\}_{\m{b} \in \m{B}}$ be a family of functionals $F_{\m{b}}:L_K(\Gamma) \to K$ such that almost all $F_{\m{b}}(x)$ are zero, then this family determines the linear map $\mathscr{D}(\cdot) = \sum_{\m{b} \in \m{B}}F_\m{b}(\cdot)  \m{b}: L_K(\Gamma) \to L_K(\Gamma)$. We ask when a  family of functionals determines a derivation. The following Theorem answers this question.

\begin{theorem}\label{maintheorem}
  Let $\Gamma$ be a row-finite graph and $K$ a commutative ring. Every derivation $\mathscr{D}$ of the Leavitt path algebra $L_K(\Gamma)$ can be described as follows:
\begin{eqnarray*}
  \mathscr{D}(v) &=&\sum_{\m{wh^*} \in \m{B}}F_{\m{wh}^*}(v) \cdot \mathbf{ad}_{\m{wh}^*}(v),\\
  \mathscr{D}(e) &=& \sum_{\m{wh^*}\in \m{B}}F_\m{wh^*}(s(e))\mathbf{ad}_{\m{wh}^*}(e) +\sum_{\substack{ \m{wh^*} \in \m{B} \\ s(\m{w}) =s(e)\\ s(\m{h}) =r(e)}}F_{\m{wh}^*}(e)\m{wh}^*,\\
  \mathscr{D}(e^*) &=&  \sum_{\m{wh^*}\in \m{B}}F_{\m{wh^*}}(s(e))\mathbf{ad}_{\m{wh^*}}(e^*)+
  \sum_{\substack{\m{wh^*} \in \m{B}  \\ s(\m{w}) =r(e)\\ s(\m{h}) =s(e)}}F_{\m{wh}^*}(e^*)\m{wh}^*,
\end{eqnarray*}
for every $v\in V$, $e\in E$, $e^*\in E^*$. Further, almost all coefficients are zero and they satisfy the following equations:
\begin{align}
  &F_\m{p}(e^*)+F_{\m{p}ff^*}(e^*)+F_{e\m{p}f}(f)=0 ,& e\m{p}f \ne 0, \label{theq3}\\
  &F_{[f\m{p}e]^*}(e^*) + F_{\m{p}^*}(f) + F_{e[\m{p}e]^*}(f) = 0, & f\m{p}e \ne 0, \label{theq6}\\
  &F_{\m{w}[f\m{h}]^*}(e^*) + F_{e\m{wh}^*}(f) = 0, & \m{w}[f\m{h}]^* \ne \m{p}ff^*, e\m{wh}^* \ne e[\m{p}e]^*,\label{theq8}
\end{align}
here $\m{p},\m{w},\m{h} \in \m{P}\cup V$ and $e,f\in E$, $s(e) =s(f)$.
\end{theorem}
\begin{proof}
Let $\D:L_K(\Gamma) \to L_K(\Gamma)$ be a $K$-linear map. Recall that (\ref{GSB}) is the Gr\"obner--Shirshov basis $\mathbf{GS}$ of $L_K(\Gamma)$. Hence, by the Leibnitz product rule, it follows that the map $\mathscr{D}$ is a derivation of $L_K(\Gamma)$ if and only if $\D(\varphi)=0$ for every $\varphi \in \mathbf{GS}$. Let $\{F_{\m{b}}\}_{\m{b} \in \m{B}}$ be a family of functionals determines the linear map $\D$. Using the Leibnitz product rule and substituting $\sum_{\m{b} \in\m{B}}F_{\m{b}}(\cdot)\m{b}$ for $\D(\cdot)$ in every $\varphi \in \mathbf{GS}$, we obtain relations for these functionals. By the straightforward computations, one can easily get these relations and just for readers convenience, we present these computations in the last section of this paper.
\end{proof}

Thus the preceding Theorem describes every element of the $K$-module $\mathbf{Der}(L_K(\Gamma))$ of derivations of the Leavitt path algebra $L_K(\Gamma)$. Our next aim is to describe a set of generators of $\mathbf{Der}(L_K(\Gamma))$.

\begin{corollary}\label{genofDer}
  The $K$-module $\mathbf{Der}(L_K(\Gamma))$ is generated by the following set of derivations:
  \begin{itemize}
  \item[(1)] $\bigcup_{\m{pq^*} \in \m{B}}\{\mathbf{ad}_{\m{pq^*}}\}$,
  \item[(2)] $\bigcup_{\substack{\m{wh^*} \in \m{M}, e \in E \\ s(\m{w}) = s(\m{h})}}\{\D_{\m{wh}^*}, \D_{ee^*}\}$, where $\D_{\m{wh}^*}(h_0) = \m{w}[\m{h}/h_0]^*,  \D_{\m{wh}^*}(w_0^*) = -[\m{w}/w_0]\m{h}^*$, and $\D_{\m{wh^*}} = 0$ otherwise,
  \item[(3)] $\bigcup_{\substack{\m{c} \in \m{P} \\ s(\m{c}) = r(\m{c})}}\{\D_\m{c}\}$, where  $\D_\m{c}(e) = \m{c}e$, for $e\in E$,  $\D_\m{c}(c_0^*) = -\m{c}/c_0$, and $\D_\m{c} = 0$ otherwise,
  \item[(4)] $\bigcup_{\substack{\m{c} \in \m{P} \\ s(\m{c}) = r(\m{c})}}\{\D_{\m{c}^*}\}$, where $\D_{\m{c}^*}(e^*) = -[\m{c}e]^*$, for $e^* \in E^*$,  $\D_{\m{c}^*}(c_0) = [\m{c}/c_0]^*$, and $\D_{\m{c}^*} = 0$ otherwise.
\end{itemize}
\end{corollary}
\begin{proof}
From Theorem \ref{maintheorem} it follows that every derivation $\D$ has the following form $\D = \mathscr{J} + \widehat{\D}$, where $\mathscr{J}$ is an inner derivation of $L_K(\Gamma)$ and $\widehat{\D}$ is defined as follows $\widehat{\D}(v) =0$, $\widehat{\D}(e) = \sum_{\substack{\m{wh^*} \in \m{B} \\ s(\m{w}) = s(e)\\ s(\m{h}) = r(e) }}F_{\m{wh^*}}(e)$ and $\widehat{\D}(e^*) = \sum_{\substack{\m{wh^*} \in \m{B}\\ s(\m{w}) = r(e)\\ s(\m{h}) = s(e) }}F_{\m{wh^*}}(e^*)$ for every $v\in V$, $e \in E$.

Using (\ref{theq3}), we get $F_{\m{w}e}(e) = - F_{\m{w}/w_0}(w_0^*) - F_{\m{w}/w_0ee^*}(w_0^*)$. Then we obtain
\begin{eqnarray*}
 \widehat{\D}(e) &=&  -\sum_{\substack{\m{w} \in \m{P} \\ s(\m{w}) = s(e)}} F_{\m{w}/w_0}(w_0^*)  \m{w}e - \sum_{\substack{ \m{w} \in \m{P} \\ s(\m{w}) = s(e) }}F_{\m{w}/w_0ee^*}(w_0^*) \m{w}e\\
 &&+ F_e(e)e+ \sum_{\substack{\m{w} \in \m{P}\cup V \\ s(\m{w}) = s(e) \\ r(\m{w}) = r(e)\\ w_z \ne e}} F_{\m{w}}(e)\m{w}   + \sum_{\substack{\m{h}^* \in \m{P^*} \cup V\\ s(\m{h}) = r(e) \\ r(\m{h}) = s(e) }} F_{\m{h^*}}(e)\m{h^*}\\
 &&+\sum_{\substack{\m{h} \in \m{P} \cup V \\ f \in E \\ s(f) = s(e)}} F_{f[\m{h}f]^*}(e) f[\m{h}f]^* + \sum_{\substack{\m{wh^*} \in \m{M} \\h_z \ne  \m{w}}}F_{\m{wh^*}}(e)\m{wh^*}.
\end{eqnarray*}

Next, by (\ref{theq6}), $F_{[\m{h}e]^*}(e^*) = - F_{[\m{h}/h_0]^*}(h_0) - F_{e[\m{h}/h_0e]^*}(h_0)$, so that
\begin{eqnarray*}
  \widehat{\D}(e^*) &=& -\sum_{\substack{\m{h} \in \m{P} \\ s(\m{h}) = s(e)}} F_{[\m{h}/h_0]^*}(h_0)[\m{h}e]^* -  \sum_{\substack{\m{h} \in \m{P} \\ s(\m{h}) = s(e)}} F_{e[\m{h}/h_0e]^*}(h_0) [\m{h}e]^*\\
  &&+F_{e^*}(e^*)e^*+  \sum_{\substack{\m{h} \in \m{P}\cup V \\ s(\m{h})=s(e) \\ r(\m{h}) = r(e)\\ h_z \ne e}} F_{\m{h^*}}(e^*)\m{h^*} +  \sum_{\substack{\m{w} \in \m{P} \cup V \\ s(\m{w}) = r(e) \\ r(\m{w}) = s(e)}} F_{\m{w}}(e^*)\m{w}\\
  &&+  \sum_{\substack{\m{w} \in \m{P} \cup V \\ f \in E \\ s(f) = s(e) }} F_{\m{w}ff^*}(e) \m{w}ff^* +  \sum_{\substack{\m{wh^*} \in \m{M}\\w_z \ne  \m{h} }} F_{\m{wh^*}} (e^*)\m{wh^*}.
\end{eqnarray*}

Then we can write
\begin{eqnarray*}
  \widehat{\D}(e) &=&  -\sum_{\substack{\m{w} \in \m{P} \\ s(\m{w}) = s(e)}} F_{\m{w}/w_0}(w_0^*)  \D_{\m{w}}(e) - \sum_{\substack{ \m{w} \in \m{P} \\ s(\m{w}) = s(e) }}F_{\m{w}/w_0ee^*}(w_0^*) \D_{\m{w}ee^*}(e) \\
   &&+ F_{e}(e)\D_{ee^*}(e)+ \sum_{\substack{\m{w} \in \m{P}\cup V \\ s(\m{w}) = s(e) \\ r(\m{w}) = r(e)\\ w_z \ne e}} F_{\m{w}}(e) \D_{\m{w}e^*}(e)   + \sum_{\substack{\m{h}^* \in \m{P^*} \cup V\\  s(\m{h}) = r(e) \\ r(\m{h}) = s(e) }} F_{\m{h^*}}(e)\D_{[e\m{h}]^*}(e)\\
   &&+\sum_{\substack{\m{h} \in \m{P} \cup V \\  f \in E \\ s(f) = s(e)}} F_{f[\m{h}f]^*}(e)\D_{f[e\m{h}f]^*}(e) +\sum_{\substack{\m{wh^*} \in \m{M}\\s(\m{w}) = s(e) \\ h_z \ne \m{w} }}F_{\m{wh^*}}(e)\D_{\m{w}[e\m{h}]^*}(e),
\end{eqnarray*}
and
\begin{eqnarray*}
  \widehat{\D}(e^*) &=& -\sum_{\substack{\m{h} \in \m{P} \\ s(\m{h}) = s(e)}} F_{[\m{h}/h_0]^*}(h_0) \D_{\m{h^*}}(e^*) -  \sum_{\substack{\m{h} \in \m{P} \\ s(\m{h}) = s(e)}} F_{e[\m{h}/h_0e]^*}(h_0) \D_{e[e\m{h}]^*}(e^*)\\
   &&+F_{e^*}(e^*)\D_{ee^*}(e^*)+  \sum_{\substack{\m{h} \in \m{P}\cup V \\ s(\m{h})=s(e) \\ r(\m{h}) = r(e)\\ h_z \ne e}} F_{\m{h^*}}(e^*) \D_{e\m{h^*}}(e^*) +  \sum_{\substack{\m{w} \in \m{P} \cup V \\ s(\m{w}) = r(e) \\ r(\m{w}) = s(e)}} F_{\m{w}}(e^*)\D_{e\m{w}}(e^*)\\
   &&+  \sum_{\substack{\m{w} \in \m{P} \cup V \\ f \in E \\ s(f) = s(e) }} F_{\m{w}ff^*}(e) \D_{e\m{w}ff^*}(e^*) +  \sum_{\substack{\m{wh^*} \in \m{M}\\s(\m{h}) = s(e)\\w_z \ne  \m{h} }} F_{\m{wh^*}} (e^*)\D_{e\m{wh^*}}(e^*).
\end{eqnarray*}

To conclude the proof, it remains to note (by (\ref{theq8})) that $F_{\m{w}[f\m{h}]^*}(e^*) = - F_{e\m{wh^*}}(f)$, $f\in E$.
\end{proof}

\begin{remark}\label{*cor}
Let us put $(e^*)^*:=e$ and $(e\cdot f)^*: = f^* \cdot e^*$ for every $e,f\in E \cup E^*$. Then Corollary \ref{genofDer} implies that $(\D_\m{c}(e))^* = -\D_{\m{c}^*}(e^*)$ and $(\D_{\m{wh}^*}(e))^* =- \D_{\m{hw^*}}(e^*)$, for every $e\in E \cup E^*$. Since $\D_{\m{c}}(v) = \D_{\m{c}^*}(v) = \D_{\m{wh^*}} = 0$ for every $v \in V$, then we can write $\D^*_{\m{c}} = - \D_{\m{c}^*}$ and $\D^*_{\m{wh^*}} = - \D_{\m{hw^*}}$.
\end{remark}

\section{The Structure of the Module of Outer Derivations}
In this section we give a presentation of the module of outer derivations (= the first Hochschild cohomology) of Leavitt path algebra.

\subsection{Calculations}

\begin{proposition}\label{beta-pi}
  Let $\Gamma = (V,E)$ be a row-finite graph and $L_K(\Gamma)$ be the Leavitt path algebra over a commutative ring $K$. Let $\{I_{\m{wh^*}}\}_{\m{wh}^*\in\m{B}}$ be a family of functionals determines an inner derivation $\mathscr{J} = \sum_{\m{wh^*} \in \mathfrak{B}}\alpha(\m{wh^*})\mathbf{ad}_{\m{wh^*}}$ of $L_K(\Gamma)$, where the $\alpha$'s are in $K$. Then for every edge $e\in E$, we have
  \begin{align}
    & I_{\m{p}}(e) = \delta_{p_z,e}\alpha(\m{p}\backslash p_z) -\delta_{p_0,e}\alpha(\m{p}/p_0)+\alpha(\m{p}e^*),  &\m{p} \in \m{P}, \label{I1}\\
    & I_{\m{p}^*}(e) = \alpha([e\m{p}]^*) - \delta_{e,\vartheta}\alpha([\m{p}e]^*), & \m{p} \in V \cup \m{P},\label{I2}\\
    & I_{f[\m{p}f]^*}(e) = \alpha(f[e\m{p}f]^*) - \delta_{e,f}\alpha([\m{p}f]^*) + \delta_{e,\vartheta}\alpha([\m{p}e]^*),& \m{p} \in V \cup \m{P},\label{I3}\\
    & I_{\m{wh^*}}(e) = \alpha(\m{w}[e\m{h}]^*) - \delta_{w_0,e}\alpha(\m{w}/w_0\m{h}^*), &\m{wh^*} \in \m{M}, \label{I4}
  \end{align}
  where $f\in E$, $s(f) = s(e)$, and if $\m{w}\in E$ then $\m{w} \ne h_z$.
  \end{proposition}
\begin{proof}
We have
\begin{eqnarray*}
  \mathscr{I}(e) &=& \sum_{\m{wh^*} \in \m{B}}\alpha(\m{wh^*})\mathbf{ad}_{\m{wh^*}}(e)  = \sum_{v\in V} \alpha(v) \Bigl(ve -ev\Bigr) \\
  && + \sum_{p\in \m{P}}\Bigl( \alpha(\m{p})\bigl(\m{p}e -e\m{p}\bigr) + \alpha(\m{p}^*)\bigl(\m{p}^*e - e \m{p}^*\bigr)\Bigr)\\
  && + \sum_{\m{wh^*}\in\m{M}}\alpha(\m{wh^*})\bigl(\m{wh^*}e - e\m{wh^*}\bigr).
\end{eqnarray*}

Adding up similar terms, we get:
\begin{eqnarray*}
\mathscr{J}(e)&=& \sum_{\m{p} \in \m{P}} \Bigl( \delta_{p_z,e}\alpha(\m{p}\backslash p_z) - \delta_{p_0,e}\alpha(\m{p}/p_0) + \alpha(\m{p}e^*) \Bigr)\m{p}\\
   && + \sum_{\m{p} \in \m{P}\cup V}\Bigl(\alpha([e\m{p}]^*)-\delta_{\vartheta,e}\alpha([\m{p}e]^*)\Bigr)\m{p}^*\\
  && + \sum_{\substack{f \in E \\ s(f) = s(e)}}\Bigl(\alpha(f[e\m{p}f]^*) - \delta_{e,f}\alpha([\m{p}f]^*) + \delta_{e,\vartheta}\alpha([\m{p}e]^*)\Bigr)f[\m{p}f]^*\\
  && +\sum_{\substack{\m{wh^*}\in \m{M} \\ \m{w} \notin E}}\bigl(\alpha(\m{w}[e\m{h}]^*) - \delta_{w_0,e}\alpha(\m{w}/w_0\m{h}^*)\bigr)\m{wh^*} + \sum_{\substack{ \m{w} \ne h_z}}\alpha(\m{w}[e\m{h}]^*)\m{wh}^*.
\end{eqnarray*}
This completes the proof.
\end{proof}

\begin{corollary}
  Let $\{I_\m{wh^*}\}_{\m{wh^*} \in \m{B}}$ be as in Proposition \ref{beta-pi}. For every edge $e\in E$, we have
    \begin{align}
    & I_{\m{p}^*}(e^*) = -\delta_{p_z,e}\alpha([\m{p}\backslash p_z]^*) +\delta_{p_0,e}\alpha([\m{p}/p_0]^*)-\alpha(e\m{p}^*),  &\m{p} \in \m{P}, \label{I1*}\\
    & I_{\m{p}}(e^*) =- \alpha(e\m{p}) + \delta_{e,\vartheta}\alpha(\m{p}e), & \m{p} \in V \cup \m{P},\label{I2*}\\
    & I_{\m{p}ff^*}(e^*) = -\alpha(e\m{p}ff^*) + \delta_{e,f}\alpha(\m{p}f) - \delta_{e,\vartheta}\alpha(\m{p}e),& \m{p} \in V \cup \m{P}, \label{I3*}\\
    & I_{\m{hw^*}}(e^*) = -\alpha(e\m{h}\m{w}^*)+ \delta_{w_0,e}\alpha(\m{h}[\m{w}/w_0]^*), &\m{hw^*} \in \m{M},\label{I4*}
  \end{align}
 where $f\in E$, $s(f) = s(e)$, and if $\m{w}\in E$ then $\m{w} \ne h_z$.
  \end{corollary}
\begin{proof}
By Remark \ref{*cor}, $(\mathbf{ad}_\m{wh^*}(e))^* = (\m{wh^*}e - e\m{wh^*})^* = e^*\m{hw^*} - \m{hw^*}e^* = -\mathbf{ad}_{\m{hw^*}}(e^*)$, so that $(\mathscr{J}(e))^* = (\sum_{\m{wh^*} \in \m{B}\setminus V} \alpha(\m{wh^*}) \mathbf{ad}_{\m{wh^*}}(e) )^* = - \mathscr{J}(e^*)$, and, by Proposition \ref{beta-pi}, the statement follows.
\end{proof}

\subsubsection{Derivations $\D_\m{c}$, $\D_{\m{c^*}}$.}~\\

We note that if $\m{c} = c_0 c_1 \cdots c_{\ell}$ is a directed cycle of a row-finite graph $\Gamma$ then we have the action of $\mathbb{Z} = \langle \zeta \rangle$ on $\m{c}$ by rotations: $\zeta^0 \m{c} = \m{c}$, $\zeta^1 \m{c} = c_1\cdots c_\ell c_0$, $\ldots$, $\zeta^\ell \m{c} = c_\ell c_0 \cdots c_{\ell-1}$, $\zeta^{\ell +1}\m{c} = \m{c}$, \textit{etc.} Set $c_{-1} = c_\ell$, $c_{-2} = c_{\ell-1}$, \textit{etc.}

\begin{corollary}\label{sysforc}
  Let $\{I_{\m{wh^*}}\}_{\m{wh^*} \in \m{B}}$ be as in Proposition \ref{beta-pi} and let $\m{c} = c_0\cdots c_\ell$ be a directed cycle of the row-finite graph $\Gamma$. For every $0 \le i \le \ell$, we have
  \[
   \begin{cases}
     \alpha(\zeta^i\m{c}) - \delta_{c_{i},e}\alpha(\zeta^{i+1}\m{c}) + \alpha(\zeta^i\m{c}ee^*) = I_{\zeta^i\m{c}e}(e), \\
    - \alpha(\zeta^i\m{c}) + \alpha(e\zeta^i\m{c}e^*) = I_{e\zeta^i\m{c}}(e), & e \ne c_{\ell-i}\\
     -\alpha(\zeta^i\m{c})+\delta_{c_i,\vartheta}\alpha(\zeta^{i+1}\m{c}) = I_{\zeta^i\m{c}/c_{i}}(c_{i}^*),\\
     -\alpha(\zeta^i\m{c}) - \alpha(\zeta^{\ell+i}\m{c}ee^*) = I_{\zeta^i\m{c}\backslash c_{i-1}ee^*}(c_{i-1}^*) , & \delta_{c_{i-1},\vartheta} = 1,\\
      \alpha (\m{p} \zeta^i\m{c} \m{p}^*) - \alpha(\m{p}/p_0\zeta^i\m{c}[\m{p}/p_0]^*) = I_{\m{p}\zeta^i\m{c}[\m{p}/p_0]^*}(p_0), & p_z \ne c_{\ell - i},\\
      \alpha(\m{p}\zeta^i \m{c} e[\m{p}e]^*) - \alpha(\m{p}/p_0\zeta^i \m{c} e[\m{p}/p_0e]^*) = I_{\m{p}\zeta^i\m{c}e[\m{p}/p_0e]^*}(p_0),
   \end{cases}
  \]
here $e \in E$, $\m{p}\in \m{P}$.
\end{corollary}
\begin{proof}
It immediately follows from (\ref{I1}), (\ref{I2*}), (\ref{I3*}), and (\ref{I4}).
\end{proof}

\begin{remark}\label{S*}
For every directed cycle $\m{c} = c_0\cdots c_\ell \in \Gamma$, where $c_0,\ldots,c_\ell \in E$ and for every $\m{wh^*} \in \m{M}$ with $s(\m{w}) = s(\m{h})$, let us consider the following series

\begin{align*}
 &  \mathsf{S}_\m{c}= \sum_{\m{p} \in \m{P} \cup V}\mathbf{ad}_{\m{pcp^*}}+  \sum_{i=1}^k \sum_{\substack{p\in \m{P}\cup V\\ e \in E\setminus\{c_i \} } } \delta_{c_i,\vartheta} \cdots \delta_{c_k,\vartheta} \Bigl(\mathbf{ad}_{\m{p}\zeta^i\m{c}\m{p}^*} - \mathbf{ad}_{\m{p}\zeta^i\m{c}e[\m{p}e]^*}\Bigr),\\
 & \mathsf{S}_{\m{wh}^*} = \sum_{\m{p} \in \m{P}\cup V}\mathbf{ad}_{\m{pw}[\m{ph}]^*},
\end{align*}
here $k +1: = \min_{1\le j \le \ell+1} \{j: \zeta^j\m{c} = \m{c}\}$.

It may be verified that, for every $x\in V$, $\mathsf{S}_\m{c}(x) = \mathsf{S}_{\m{wh^*}}(x) =0$. Thus, by Remark \ref{*cor}, the equality $(\mathsf{S}_\m{c}(x))^*: = - \mathsf{S}_\m{c^*}(x^*)$ and $(\mathsf{S}_\m{wh^*}(x))^*: = - \mathsf{S}_\m{hw^*}(x^*)$ are well defined, for every $x \in L_K(\Gamma)$.
\end{remark}

\begin{lemma}\label{Dc=ad}
 Let $\Gamma = (V,E)$ be a row-finite graph and $L_K(\Gamma)$ be the Leavitt path algebra over a commutative ring $K$. Let $\m{c} = c_0 \cdots c_\ell \in \Gamma$ be a directed cycle, where $c_0,\ldots,c_\ell \in E$. Assume that $\D_\m{c}(x) = \sum_{\m{wh^*}\in \m{B}}\alpha(\m{wh^*})\mathbf{ad}_\m{wh^*}(x)$ for every $x\in L_K(\Gamma)$, then $\D_{\m{c}}(x) = \mathsf{S}_\m{c}(x)$ for every $x\in L_K(\Gamma)$ if at least one of the edges $c_0,\cdots, c_\ell$ is not special, and where $k +1: = \min_{1\le j \le \ell+1} \{j: \zeta^j\m{c} = \m{c}\}$.
\end{lemma}
\begin{proof}
From Corollary \ref{genofDer}, we already know that the family $\{F_{\m{c}e}(e) = 1, F_{c_1\cdots c_\ell}(c_0^*) = -1\}_{e\in E}$ of functionals determines the derivation $\D_\m{c}$. Let us assume that $\D_\m{c}(x) = \sum_{\m{wh^*}\in \m{B}}\alpha(\m{wh^*})\mathbf{ad}_\m{wh^*}(x)$ for every $x\in L_K(\Gamma)$. By Corollary \ref{sysforc},
  \[
   \begin{cases}
     \alpha(\zeta^i\m{c}) - \delta_{c_{i},e}\alpha(\zeta^{i+1}\m{c}) + \alpha(\zeta^i\m{c}ee^*) = F_{\zeta^i\m{c}e}(e), \\
    - \alpha(\zeta^i\m{c}) + \alpha(e\zeta^i\m{c}e^*) = F_{e\zeta^i\m{c}}(e), & e \ne c_{\ell-i}\\
     -\alpha(\zeta^i\m{c})+\delta_{c_i,\vartheta}\alpha(\zeta^{i+1}\m{c}) = F_{\zeta^i\m{c}/c_{i}}(c_{i}^*),\\
     -\alpha(\zeta^i\m{c}) - \alpha(\zeta^{\ell+i}\m{c}ee^*) = F_{\zeta^i\m{c}\backslash c_{i-1}ee^*}(c_{i-1}^*) , & \delta_{c_{i-1},\vartheta} = 1, \\
      \alpha (\m{p} \zeta^i\m{c} \m{p}^*) - \alpha(\m{p}/p_0\zeta^i\m{c}[\m{p}/p_0]^*) = F_{\m{p}\zeta^i\m{c}[\m{p}/p_0]^*}(p_0), & p_z \ne c_{\ell - i},\\
      \alpha(\m{p}\zeta^i \m{c} e[\m{p}e]^*) - \alpha(\m{p}/p_0\zeta^i \m{c} e[\m{p}/p_0e]^*) = F_{\m{p}\zeta^i\m{c}e[\m{p}/p_0e]^*}(p_0),
   \end{cases}
  \]
here $0 \le i \le \ell$ and $e \in E$, $\m{p}\in \m{P}$, so that
  \[
   \begin{cases}
     \alpha(\zeta^i\m{c}) - \delta_{c_{i},e}\alpha(\zeta^{i+1}\m{c}) + \alpha(\zeta^i\m{c}ee^*) = \delta_{i,0}, \\
    - \alpha(\zeta^i\m{c}) + \alpha(e\zeta^i\m{c}e^*) =0, & e \ne c_{k-i}\\
     -\alpha(\zeta^i\m{c})+\delta_{c_i,\vartheta_i}\alpha(\zeta^{i+1}\m{c}) = -\delta_{i,0},\\
     -\alpha(\zeta^i\m{c}) - \alpha(\zeta^{k+i}\m{c}ee^*) = 0 , & \delta_{c_{i-1},\vartheta} = 1,\\
      \alpha (\m{p} \zeta^i\m{c} \m{p}^*)= \alpha(\zeta^i\m{c}), & p_z \ne c_{k - i},\\
      \alpha(\m{p}\zeta^i \m{c} e[\m{p}e]^*) = \alpha(\zeta^i \m{c} ee^*),
   \end{cases}
  \]
where $k+1= \min_{1 \le j \le \ell+1}\{j: \zeta^j\m{c} = \m{c} \}.$

It is easy to see that from the equations $-\alpha(\zeta^i\m{c})+\delta_{c_i,\vartheta}\alpha(\zeta^{i+1}\m{c}) = -\delta_{i,0}$, $0 \le i \le k$ it follows that
\begin{align*}
&\alpha(\m{c}) = 1 + \delta_{c_0,\vartheta}\alpha(\zeta \m{c}) \\
&\alpha(\zeta \m{c}) = \delta_{c_1,\vartheta} \alpha(\zeta^2\m{c}) = \ldots = \delta_{c_0,\vartheta}\cdots \delta_{c_k,\vartheta}\alpha(\m{c}),
\end{align*}
then $\alpha(\m{c}) = 1 + \delta_{c_0,\vartheta}\cdots \delta_{c_k,\vartheta}\alpha(\m{c})$. It implies that if all edges $c_0, \ldots,c_k$ are special we then get a contradiction (namely $1=0$). Assume that at least one of the edges $c_0,\ldots,c_k$ is not special. Then $\alpha(\m{c}) = 1$ and $\alpha(\zeta^i\m{c}) = \delta_{c_i,\vartheta}\cdots \delta_{c_k,\vartheta}$ for $1 \le i \le k$.

Let $i \ne 1$ and the edge $c_{i-1}$ be special. Consider the equation $\alpha(\zeta^i\m{c}) + \alpha(\zeta^{k+i}\m{c}ee^*) = 0$. We can write $ \delta_{c_i,\vartheta}\cdots \delta_{c_k,\vartheta} + \alpha(\zeta^{i-1}\m{c}ee^*)=0$, because $\zeta^{k+i}\m{c} = \zeta^{k+1+i-1}\m{c} = \zeta^{i-1}\m{c}$. Since $\delta_{c_{i-1},\vartheta} =1$, $\alpha(\zeta^{i-1}\m{c}) = \delta_{c_{i-1},\vartheta }\cdots \delta_{c_k,\vartheta }$, then the preceding equation gives $\alpha(\zeta^{i-1}\m{c}) + \alpha(\zeta^{i-1}\m{c}ee^*) =0.$ Assume that $c_0$ is special. Then $\alpha(\zeta^1\m{c}) = \delta_{c_1,\vartheta} \cdots \delta_{c_k,\vartheta } = 0$. Consider the equation $\alpha(\zeta^1\m{c}) + \alpha(\m{c}ee^*) = 0$. It follows that $\alpha(\m{c}ee^*) = 0$, \textit{i.e.,} $1 + \alpha(\m{c}ee^*) = 1$. We have thus shown that every equation $\alpha(\zeta^i\m{c}) + \alpha(\zeta^{k+i}\m{c}ee^*) = 0$ is equivalent to the equation $\alpha(\zeta^i\m{c}) - \delta_{c_{i},e}\alpha(\zeta^{i+1}\m{c}) + \alpha(\zeta^i\m{c}ee^*) = \delta_{i,0}$ for every $0 \le i \le k$.

Let us consider the equations $\alpha(\zeta^i\m{c}) - \delta_{c_{i},e}\alpha(\zeta^{i+1}\m{c}) + \alpha(\zeta^i\m{c}ee^*) = \delta_{i,0}$, $0 \le i \le k$. We have $\alpha(\m{c}ee^*) = 0$, for every $e \in E$, and $\alpha(\zeta^i\m{c}c_ic_i^*) = \alpha(\zeta^{i+1}\m{c}) - \alpha(\zeta^i\m{c})$ for every $1 \le i \le k$. Let the edge $c_i$ be special, then we have to put $\alpha(\zeta^i\m{c}c_ic_i^*) := 0$. Thus
\[
\alpha(\zeta^{i+1}\m{c}) - \alpha(\zeta^i\m{c}) = \delta_{c_{i+1},\vartheta }\cdots \delta_{c_k,\vartheta } - \delta_{c_{i+1},\vartheta }\cdots \delta_{c_k,\vartheta } =0,
\]
because $\delta_{c_{i},\vartheta } = 1$. Assume that the edge $c_i$ is not special. Then $\alpha(\zeta^i\m{c})=\delta_{c_{i},\vartheta}\cdots \delta_{c_k,\vartheta} = 0$, hence $\alpha(c_i\zeta^{i+1}\m{c}c_i^*) = \alpha(\zeta^{i+1}\m{c})$, because $\zeta^i\m{c}c_i = c_i \zeta^{i+1}\m{c}$. Thus, we obtain
\begin{align*}
  &\alpha(\m{pcp}^*)= \alpha(\m{c}) =1,\\
  & \alpha(\m{p}\zeta^i\m{c}\m{p}^*) =\alpha(\zeta^i\m{c})= \delta_{c_i,\vartheta} \cdots \delta_{c_k,\vartheta},\\
  & \alpha(\m{p}\zeta^i\m{c}e[\m{p}e]^*) = - \delta_{c_i,\vartheta} \cdots \delta_{c_k,\vartheta}, & e \ne c_{i-1},\\
  &\delta_{c_0,\vartheta} \cdots \delta_{c_k,\vartheta} = 0,
\end{align*}
for $1 \le i \le k$. This completes the proof.
\end{proof}

\begin{corollary}\label{Dc-outer}
  For every directed cycle $\m{c} \in \Gamma$, the derivations $\D_{\m{c}}$, $\D_{\m{c^*}}$ are outer.
\end{corollary}
\begin{proof}
  It immediately follows from Lemma \ref{Dc=ad}, Remarks \ref{*cor}, \ref{S*}.
\end{proof}

\subsubsection{Derivations $\D_{\m{wh^*}}$}

\begin{lemma}\label{Dwh=ad}
  Let $\Gamma = (V,E)$ be a row-finite graph and $L_K(\Gamma)$ be the Leavitt path algebra over a commutative ring $K$. Let $w_0,\cdots,w_m,h_0,\ldots,h_n \in E$ such that $s(w_0) = s(h_0)$, $\m{wh^*} = w_0\cdots w_m [h_0\cdots h_n]^* \in \m{M}$ and $\m{wh^*} \ne w_0h_0^*$. If $\D_{\m{wh}^*}(x) = \sum_{\m{pq^*}\in \m{B}} \alpha(\m{pq^*}) \mathbf{ad}_{\m{pq^*}}(x)$ for every $x\in L_K(\Gamma)$, then $\D_{\m{wh^*}}(x) = \mathsf{S}_{\m{wh^*}} (x)$ for every $x\in L_K(\Gamma)$.
\end{lemma}
\begin{proof} We start with the following notation. Let $\m{p,q \in P}$. We write $\m{q} \not \succcurlyeq  \m{p}$ if $\m{q} \ne \m{q'p}$ for some $\m{q'} \in V \cup \m{P}.$

(1) Let $m \ge n$. By Corollary \ref{genofDer}, the following family $\{F_{\m{w}[\m{h}/h_0]^*}(h_0)=  1, F_{\m{w}/w_0\m{h}^*}(w_0^*) = -1  \}$ determines the derivation $\D_{\m{wh^*}}$. Let $\D_{\m{wh}^*}(x) = \sum_{\m{pq^*}\in \m{B}}\alpha(\m{pq^*})\mathbf{ad}_{\m{pq^*}}(x)$ for every $x\in L_K(\Gamma)$. By (\ref{I4}) and (\ref{I4*}), we obtain
\begin{align*}
     &F_{\m{w}[\m{h}/h_0]^*}(h_0) = \alpha(\m{wh^*}) - \delta_{w_0,h_0}\alpha(\m{w}/w_0[\m{h}/h_0]^*) = 1,\\
     &F_{\m{w}/w_0\m{h}^*}(w_0^*) = -\alpha(\m{wh^*}) + \delta_{w_0,h_0}\alpha(\m{w}/w_0[\m{h}/h_0]^*) =- 1,
\end{align*}
  and therefore $\alpha(\m{wh^*}) = 1 + \delta_{w_0,h_0}\alpha(\m{w}/w_0[\m{h}/h_0]^*)$. Using (\ref{I4}), we get $\alpha(\m{pw}[\m{ph}]^*) - \alpha(\m{p}/p_0 \m{w}[\m{p}/p_0\m{h}]^*) = F_{\m{pw}[\m{p}/p_0\m{h}]}(p_0)$. Hence \[
  \alpha(\m{pw}[\m{ph}]^*) = \alpha(\m{p}/p_0 \m{w}[\m{p}/p_0\m{h}]^*) = \cdots = \alpha(\m{wh^*}), \qquad \m{p} \in \m{P}.
  \]
From (\ref{I4}) it follows that
\begin{multline*}
  \alpha(w_i\cdots w_m[h_i \cdots h_n]^*) - \delta_{w_i,h_i}\alpha(w_{i+1}\cdots w_m[h_{i+1}\cdots h_n]^*) \\= F_{w_i\cdots w_m[h_{i+1}\cdots h_n]^*}(h_i)
\end{multline*}
and
\begin{multline*}
  \alpha(ew_i\cdots w_m[eh_i \cdots h_n]^*) - \alpha(w_i\cdots w_m[h_i \cdots h_n]^*) \\= F_{ew_i\cdots w_m[h_{i}\cdots h_n]^*}(e),
\end{multline*}
for $e\in E$, $1 \le i \le n-2$.  Therefore, we get $\alpha(w_i\cdots w_m[h_i\cdots h_n]^*) = \delta_{w_i,h_i}\alpha(w_{i+1}\cdots w_n[h_{i+1}\cdots h_n]^*)$. So, we obtain
\begin{align*}
  & \alpha(\m{w}/w_0[\m{h}/h_0]^*) = \delta_{w_1 \cdots w_{n-1},h_1 \cdots h_{n-1}} \alpha(w_n\cdots w_mh_n^*),\\
  & \alpha(\m{q}w_i\cdots w_m[\m{q}h_i \cdots h_n]^*)  =\alpha(w_i\cdots w_m[h_i \cdots h_n]^*),
\end{align*}
for every $1 \le i \le n-2$ and $\m{q \in P}$ such that $\m{q} \not \succcurlyeq w_0\cdots w_{i-1}$.

 Further, by (\ref{I1}), we have
\begin{eqnarray*}
F_{w_n\cdots w_m}(h_n) &=& \delta_{w_m,h_n}\alpha(w_n\cdots w_{m-1}) \\
 &&- \delta_{w_n,h_n}\alpha(w_{n+1}\cdots w_m)+ \alpha(w_n\cdots w_mh_n^*)\\
 &=& \delta_{n,0}.
\end{eqnarray*}

We note that if $w_0 = h_0, \ldots, w_n = h_n$, then $s(\m{w}_{n+1}) = r(\m{w}_{n+1})$, and if $w_0 = h_0, \ldots, w_{n-1} = h_{n-1}, w_m = h_n$, then $s(\m{w}_n) = r(\m{w}_n)$, here $\m{w}_{n+1} := w_{n+1}\cdots w_m$ and $\m{w}_n :=  w_n\cdots w_{m-1}$. Thus we can write
\begin{multline*}
 \alpha(\m{wh^*}) = 1 - \delta_{w_0\cdots w_n,\m{h}}\alpha(\m{w}_{n+1})+ \delta_{w_0,h_0}\cdots \delta_{w_{n-1},h_{n-1}}\delta_{w_m,h_n}\alpha(\m{w}_n),\\
   \alpha(w_{i+1}\cdots w_m[h_{i+1}\cdots h_n]^*) = \delta_{w_{i+1},h_{i+1}}\cdots \delta_{w_n,h_n}\alpha(\m{w}_{n+1})\\
  -\delta_{w_{i+1},h_{i+1}}\cdots \delta_{w_{n-1},h_{n-1}}\delta_{w_m,h_n}\alpha(\m{w}_n),
\end{multline*}
where $1 \le i \le m-1$.

Just for convinces, let us write $\zeta^i\m{w}_{n+1} = w_{n+i+1} \cdots w_{m+i}$, for $0\le i \le m-n-1$. Then, using (\ref{I2}), we get
\[
\delta_{w_{n+i+1}, \vartheta_{n+i+1}} \alpha(\zeta^{i+1}\m{w}_{n+1}) -\alpha(\zeta^i \m{w}_{n+1})= F_{\zeta^i\m{w}_{n+1}/w_{n+i+1}}(w_{n+i+1}^*) = 0,
\]
for $0 \le i \le m-n-1$. It follows that if at least one of the edge $w_{n+1},\ldots,w_m$ is not special then $\alpha(\zeta^{i}\m{w}_{n+1}) = 0 $ for all $0 \le i \le m-n-1$. Assume that all edges $w_m, \ldots, w_n$ are special. By Corollary \ref{sysforc}, we get
\[
   \begin{cases}
     \alpha(\zeta^i\m{w}_{n+1}) - \delta_{w_{n+i+1},e}\alpha(\zeta^{i+1}\m{w}_{n+1}) + \alpha(\zeta^i\m{w}_{n+1}ee^*) = F_{\zeta^i\m{w}_{n+1}e}(e) \\
      - \alpha(\zeta^i\m{w}_{n+1}) + \alpha(e\zeta^i\m{w}_{n+1}e^*) = F_{e\zeta^i\m{w}_{n+1}}(e), \qquad E \ni e \ne w_{n+i+1},\\
     -\alpha(\zeta^i\m{w}_{n+1})+\alpha(\zeta^{i+1}\m{w}_{n+1}) = F_{\zeta^i\m{w}_{n+1}/w_{n+i+1}}(w_{n+i+1}^*)\\
     -\alpha(\zeta^i\m{w}_{n+1}) - \alpha(\zeta^{m-n+1+i}\m{w}_{n+1}ee^*) =F_{[\zeta^i\m{w}_{n+1}\backslash w_{m+i}]ee^*}(w_{m+i}^*) \\
      \alpha (\m{p} \zeta^i\m{w}_{n+1} \m{p}^*) - \alpha(\m{p}/p_0\zeta^i\m{w}_{n+1}[\m{p}/p_0]^*) = F_{\m{p}\zeta^i\m{w}_{n+1}[\m{p}/p_0]^*}(p_0)\\
      \alpha(\m{p}\zeta^i \m{w}_{n+1} e[\m{p}e]^*) - \alpha(\m{p}/p_0\zeta^i \m{w}_{n+1} e[\m{p}/p_0e]^*) = F_{\m{p}\zeta^i\m{w}_{n+1}e[\m{p}/p_0e]^*}(p_0),
   \end{cases}
\]
here $0 \le i \le m-n-1$ and $e \in E$, $\m{p}\in \m{P}$. We have
\begin{align*}
   &F_{\zeta^i \m{w}_{n+1} e}(e) = \delta_{i,0}\delta_{n,0}\delta_{e,h_n}, \\
   &F_{e\zeta^i\m{w}_{n+1} } (e) = 0 , \quad \mbox{where $e \ne w_{n+i+1},$}\\
   &F_{\zeta^i\m{w}_{n+1}/w_{n+i+1}}(w_{n+i+1}^*) = 0,\\
   & F_{[\zeta^i\m{w}_{n+1}\backslash w_{m+i}]ee^*}(w_{m+i}^*) = 0,\\
   & F_{\m{p}\zeta^i\m{w}_{n+1}[\m{p}/p_0]^*}(p_0) = 0, \quad \mbox{whenever $\m{p}\zeta^i\m{w}_{n+1}[\m{p}/p_0]^* \ne \m{w[h/h_0]^*}$},\\
   & F_{\m{p}\zeta^i\m{w}_{n+1}e[\m{p}/p_0e]^*}(p_0) = 0,
\end{align*}
for every $0 \le i \le m-n-1$. Hence, we obtain
\begin{align*}
  & \alpha(\m{w}_{n+1}) = \alpha(\zeta \m{w}_{n+1}) = \cdots = \alpha(\zeta^{m-n-1}\m{w}_{n+1})\\
  & \alpha(\m{p} \zeta^i\m{w}_{n+1}e [\m{p}e]^*) =  \alpha(\zeta^i \m{w}_{n+1}ee^*) = - \alpha(\zeta^{m-n+i+i}\m{w}_{n+1}) \\
  & \alpha(\m{p}\zeta^i \m{w}_{n+1}\m{p}^*) = \alpha(\zeta^i \m{w}_{n+1}),
\end{align*}
in the last equation we have assumed that $\m{p}\zeta^i \m{w}_{n+1}\m{p}^* \ne \m{wh^*}$. Similarly, for the the directed cycle $\m{w}_n = w_n\cdots w_{m-1}$, one can get
\begin{align*}
  & \alpha(\m{w}_{n}) = \alpha(\zeta \m{w}_{n}) = \cdots = \alpha(\zeta^{m-n}\m{w}_{n})\\
  & \alpha(\m{p} \zeta^j\m{w}_{n}e [\m{p}e]^*) =  \alpha(\zeta^j \m{w}_{n}ee^*) = - \alpha(\zeta^{m-n-1+j}\m{w}_{n}) \\
  & \alpha(\m{p}\zeta^j \m{w}_{n}\m{p}^*) = \alpha(\zeta^j \m{w}_{n}),
\end{align*}
here $0 \le j \le m-n-1$, $e \in E$, $\m{p}\in \m{P}$, all edges $w_n, \ldots w_{m-1}$ are special, and $\m{p}\zeta^j \m{w}_{n}\m{p}^*\ne \m{wh^*}$ in the last equation.

Denote $\delta_{w_0\cdots w_n, \m{h}}\delta_{w_{n+1},\vartheta }\cdots \delta_{w_m,\vartheta }$ and $\delta_{w_0 \cdots w_{n-1},h_0 \cdots h_{n-1}} \delta_{w_m,h_n}\delta_{w_{n+1},\vartheta }\cdots \delta_{w_m,\vartheta }$ by $\Delta_{n+1}$ and $\Delta_n$, respectively. Then we get

\begin{equation}\label{sysforwh}
 \begin{cases}
   \alpha(\m{pw}[\m{ph}^*]) = 1+ \Delta_{n+1}\alpha - \Delta_{n}\beta,\\
   \delta_{w_0,h_0}\cdots \delta_{w_i,h_i}\alpha(\m{p}w_{i+1}\cdots w_m[\m{p}h_{i+1}\cdots h_n]^*), \\
   \phantom{AAAAAAAAAAAAAAAAAAA}= \Delta_{n+1}\alpha - \Delta_{n}\beta,\\
   \alpha(\m{p}\zeta^i\m{w}_{k}e[\m{p}e]^*) = - \Delta_{k}\alpha(\m{w}_{k}),\\
   \alpha(\zeta\m{w}_{k}) = \cdots = \alpha(\zeta^{n-m-1}\m{w}_{k}) = \Delta_{k}\alpha(\m{w}_{k}),\\
    \alpha(\m{p}\zeta^i\m{w}_{k}\m{p}^*) =  \Delta_{k}\alpha(\m{w}_{k}),
  \end{cases}
\end{equation}
here $ 0 \le i,j \le m-n-1$, $k\in \{n,n+1\}$, $\alpha,\beta \in K$, $\m{p} \in \m{P}$ and $\m{p}\zeta^i\m{w}_{k}\m{p}^* \ne \m{qw[qh]^*}$, $\m{p} \not \succcurlyeq \m{q}$, in the last equality. Thus we get

\begin{eqnarray*}
\D_\m{wh^*}(x)&=& \sum_{\m{p} \in \m{P}\cup V}\mathbf{ad}_{\m{pw}[\m{ph}]^*}(x)\\
  && + \Delta_{n+1} \alpha \sum_{i=0}^{m-n-1}\sum_{\m{p} \in \m{P}\cup V } \sum_{e \in E} \Bigl(\mathbf{ad}_{\m{p}\zeta^i\m{w}_{n+1}\m{p}^*}(x) - \mathbf{ad}_{\m{p}\zeta^{i}\m{w}_{n+1}e[\m{p}e]^*}(x) \Bigr)\\
  && +  \Delta_{n} \beta \sum_{i=0}^{m-n-1}\sum_{\m{p} \in \m{P} \cup V} \sum_{e\in E} \Bigl(\mathbf{ad}_{\m{p}\zeta^i\m{w}_{n}\m{p}^*}(x) - \mathbf{ad}_{\m{p}\zeta^{i}\m{w}_{n}e[\m{p}e]^*}(x) \Bigr).
\end{eqnarray*}

By the straightforward verification, it is easy to check that
\begin{align*}
  &\Delta_{n+1} \alpha \sum_{i=0}^{m-n-1}\sum_{\m{p} \in \m{P}\cup V } \sum_{e \in E} \Bigl(\mathbf{ad}_{\m{p}\zeta^i\m{w}_{n+1}\m{p}^*}(x) - \mathbf{ad}_{\m{p}\zeta^{i}\m{w}_{n+1}e[\m{p}e]^*}(x) \Bigr) = 0\\
  &\Delta_{n} \beta \sum_{i=0}^{m-n-1}\sum_{\m{p} \in \m{P} \cup V} \sum_{e\in E} \Bigl(\mathbf{ad}_{\m{p}\zeta^i\m{w}_{n}\m{p}^*}(x) - \mathbf{ad}_{\m{p}\zeta^{i}\m{w}_{n}e[\m{p}e]^*}(x) \Bigr) =0
\end{align*}
for every $x \in V\cup E\cup E^*$ and the statement follows.

(2) Let $m \le n$. By Remark \ref{*cor} and Remark \ref{S*}, $(\D_{\m{wh^*}}(x))^* = -\D_{\m{hw^*}}(x^*)$ and the statement follows.
\end{proof}

\begin{corollary}\label{Dwh=sum}
  Let $\m{wh^* \in M}$, $s(\m{w}) = s(\m{h})$, and there is a finite number of edges $e\in E$ such that $r(e) = s(\m{w})$. Then
  \[
  \D_{\m{wh^*}} \equiv \sum_{e\in E} \D_{e\m{w}[e\m{h}]^*} \bmod \mathbf{ad}(L_K(\Gamma)).
  \]
\end{corollary}
\begin{proof}
  By Lemma \ref{Dwh=ad}, $\D_{\m{wh^*}}(x) =\sum_{\m{p} \in \m{P} \cup V} \mathbf{ad}_{\m{pw[ph]^*}}(x)$ and $\D_{e\m{w}[e\m{h}]^*}(x) =\sum_{\m{p} \in \m{P} \cup V} \mathbf{ad}_{\m{p}e\m{w}[\m{p}e\m{h}]^*}(x)$ for every $x \in L_K(\Gamma)$, so that $\D_{\m{wh^*}} - \sum_{e\in E} \D_{e\m{w}[e\m{h}]^*} = \mathbf{ad}_{\m{wh^*}}$ and the statement follows.
\end{proof}

\begin{corollary}\label{Dc=Dkc}
Let $\m{c} = c_0 c_1 \cdots c_\ell$ be a directed cycle of $\Gamma$, $k+1: = \min_{1 \le i \le \ell+1}\{i: \zeta^i\m{c} = \m{c}\}$ and the edge $c_k$ is special. Then
\[
 \D_{\m{c}} \equiv \D_{\zeta^k\m{c}} - \sum_{e \in E} \D_{\zeta^k \m{c} ee^*} \bmod \mathbf{ad}(L_K(\Gamma)),
\]
if at least one of the edges $c_0,\ldots, c_k$ is not special and the set $\{ \m{p} \in \m{P}: r(\m{p}) = s(c_0),  p_z \ne c_k \}$ is finite.
\end{corollary}
\begin{proof}
By Lemma \ref{Dc=ad}, $\D_{\m{c}}(x) = \mathsf{S}_\m{c}(x)$ for every $x \in L_K(\Gamma)$, where
\[
\mathsf{S}_\m{c} =  \sum_{\m{p} \in \m{P} \cup V}\mathbf{ad}_{\m{pcp^*}}+  \sum_{i=1}^k \sum_{\substack{p\in \m{P}\cup V\\ e \in E\setminus\{c_i \} } } \delta_{c_i,\vartheta} \cdots \delta_{c_k,\vartheta} \Bigl(\mathbf{ad}_{\m{p}\zeta^i\m{c}\m{p}^*} - \mathbf{ad}_{\m{p}\zeta^i\m{c}e[\m{p}e]^*}\Bigr),
\]
$\delta_{c_0,\vartheta } \cdots \delta_{c_k, \vartheta } =0$, and $k+1 : = \min_{1\le j \le \ell +1} \{j: \zeta^j\m{c} = \m{c}\}$.

For every $0 \le i \le k$, denote by $\widetilde{\mathsf{S}}_{\zeta^i\m{c}}$ the series $\mathsf{S}_{\zeta^i\m{c}}  - \sum_{\m{p} \in \m{P}  \cup V } \mathbf{ad}_{\m{p}\zeta^i\m{c}\m{p^*} }$ and let $\mathsf{S}_i: = \sum_{ \substack{ \m{p} \in \m{P}\cup V\\ e \in E\setminus\{c_i \} } } \Bigl(\mathbf{ad}_{\m{p}\zeta^i\m{c}\m{p}^*}(x) - \mathbf{ad}_{\m{p}\zeta^i\m{c}e[\m{p}e]^*}(x) \Bigr)$. Then we can write
\[
\widetilde{\mathsf{S} }_{\zeta^i\m{c}} = \mathsf{S}_i + \delta_{c_{i+1},\vartheta } \cdots \delta_{c_{i-1},\vartheta } \mathsf{S}_{i+1}  + \cdots + \delta_{c_{i-1},\vartheta } \mathsf{S}_{i-1}.
\]

By $\delta_{c_0,\vartheta} \cdots \delta_{c_k, \vartheta_{k}} =0$,
\begin{multline*}
 \delta_{c_{i-1},\vartheta } \widetilde{\mathsf{S} }_{\zeta^{i-1}\m{c}} = \delta_{c_{i-1},\vartheta } \mathsf{S}_{i-1} + \delta_{c_{i-1},\vartheta } \delta_{c_{i+1,\vartheta } } \cdots \delta_{c_{i-2},\vartheta } \mathsf{S}_{i+1}\\ + \cdots +\delta_{c_{i-1}, \vartheta}  \delta_{c_{i-2},\vartheta}  \mathsf{S}_{i-2},
\end{multline*}
so that $\widetilde{\mathsf{S} }_{\zeta^i\m{c}} = \mathsf{S}_i + \delta_{c_{i-1},\vartheta} \widetilde{\mathsf{S} }_{\zeta^{i-1}\m{c}}$. It follows that
\[
 \D_{\zeta^i\m{c}}(x)  = \sum_{\m{p} \in \m{P} \cup V} \mathbf{ad}_{\m{p}\zeta^i \m{c}\m{p}^*}(x) + \delta_{c_{i-1},\vartheta} \D_{\zeta^{i-1}\m{c}} (x)  - \delta_{c_{i-1},\vartheta} \sum_{ \substack{ \m{p} \in \m{P} \cup V \\ e \in E \setminus \{ c_{i-1} \} } } \mathbf{ad}_{\m{p}\zeta^i\m{c} e[\m{p}e]^*}(x)
\]
for every $x \in L_K(\Gamma)$. Taking into account Lemma \ref{Dwh=ad}, we obtain
\[
 \D_{\zeta^i \m{c}}(x) + \sum_{e \in E\setminus \{c_{i-1}\}} \D_{\zeta^{i-1}\m{c} ee^*}(x) -\delta_{c_{i-1}, \vartheta} \D_{\zeta^{i-1}\m{c}}(x)  = \sum_{\m{p} \in \m{P} \cup V}  \mathbf{ad}_{\m{p}\zeta^i\m{c}\m{p}^*}(x),
\]
because $\delta_{c_0,\vartheta} \cdots \delta_{c_k, \vartheta} =0$, and the statement follows.
\end{proof}

\subsubsection{Derivations $\D_{ef^*}$}~\\

We start with some notations. \textit{A special path} $\m{\overline p}: =  \bar e_1  \cdots \bar e_n$ in $\Gamma$ is a sequence of special edges $\bar e_1, \ldots, \bar e_n \in \vartheta(V)$ such that $r(\bar e_i) = s(\bar e_{i+1})$ or $r(\bar e_i) = r(\bar e_{i+1})$ for $i=1,\ldots, n-1$. Set $s(\m{\overline p}): = s(\bar e_1)$ and $r(\m{\overline p}): = r(\bar e_n)$. Denote by $ \m{\overline P}$ the set of all special paths of $\Gamma$. Further, set $V \supseteq \widetilde{V} : = \{v \in V: r^{-1}(v) \cap \vartheta(V) = \varnothing\}$.

Next, let $\m{\overline p}: =  \bar e_1  \cdots \bar e_n$ be a special path and $v\in V$. We write $v \in \m{\overline p}$ if $v = s(\bar e_i)$ or $v = r(\bar e_i)$ for some $1 \le i \le n$. If $E \ni e \ne e_i$ for every $1 \le i \le n$ we then write $e \not \subset \m{\bar p}$.

Set $\mathsf{supp}_e(v):=  \{ \m{\overline p}  \in \m{\overline P}: v\in \m{\overline p}, e \not \subset \m{\overline p } \}$ for every fixed $e\in E$ and $v \in V$. If $\mathsf{supp}_e(s(e)) \cap \mathsf{supp}_e(r(e)) = \varnothing$ then there is no special path $\m{\overline p}$ with $s(\m{\overline p}) = s(e)$, $r(\m{\overline p}) = r(e)$ or $s(\m{\overline p})= r(e)$, $r(\m{\overline p}) = s(e)$.

\begin{lemma}
Let $e\in E$ be an edge with $s(e) \ne r(e)$. Suppose that $\D_{ee^*}(x) = \sum_{\m{wh^* \in B}} \alpha(\m{wh^*}) \mathbf{ad}_{\m{wh^*}}(x)$ for every $x\in V \cup E \cup E^*$.

\begin{itemize}
  \item[(1)] If the edge $e$ is not special then $\D_{ee^*} = \sum_{\m{p} \in \m{P}\cup V} \mathbf{ad}_{\m{p}e[\m{p}e]^*}$.\\
  \item[(2)] If the edge $e$ is special and $\mathsf{supp}_e(s(e)) \cap \mathsf{supp}_e(r(e)) = \varnothing$ then
  \[
      \D_{ee^*} = \alpha \Bigl( \mathsf{S}_{s(e)} + \sum_{\substack{ u\in \mathrm{supp}(s(e)) \\ u \ne r(e)} } \mathbf{ad}_{u}\Bigr)+ \beta \Bigl(\mathsf{S}_{r(e)} +\sum_{\substack{ u\in \mathrm{supp}(r(e)) \\ u \ne r(e)} } \mathbf{ad}_{u} \Bigr),
  \]
where $\alpha,\beta \in K$, $\alpha - \beta =1$, every $\sum_{\substack{ u\in \mathrm{supp}_e(s(e)) \\ u \ne r(e)} } \mathbf{ad}_{u}$, $\sum_{\substack{ u\in \mathrm{supp}_e(s(e)) \\ u \ne s(e)} } \mathbf{ad}_{u}$ does not contain equal terms, and
\[
  \mathsf{S}_v: = \sum_{\m{p} \in \m{P}\cup V} \sum_{\substack{g \in E \\ r(g) \in \mathsf{supp}_e(v) } } \mathbf{ad}_{\m{p}g[\m{p}g]^* } - \sum_{\substack{f \in E \\ s(f) \in \mathrm{supp}_e(v) } } \mathbf{ad}_{\m{p}f[\m{p}f]^* }.
\]
And if $\mathsf{supp}_e(s(e)) \cap \mathsf{supp}_e(r(e)) \ne \varnothing$ then the derivation $\D_{ee}^*$ cannot be presented as any series of form $ \sum_{\m{wh^* \in B}} \alpha(\m{wh^*}) \mathbf{ad}_{\m{wh^*}}.$
\end{itemize}

\end{lemma}

\begin{proof}
By Corollary \ref{genofDer}, the family $\{F_{e}(e) = 1, F_{e^*}(e^*)=-1 \}$ determines the derivation $\D_{ee}^*$. Using (\ref{I1}), (\ref{I4}), we get
\begin{eqnarray}\label{sysef}
\begin{cases}
 \alpha(s(e)) - \alpha(r(e)) + \alpha(ee^*) = 1,\\
 \alpha(s(f)) - \alpha(r(f)) + \alpha(ff^*) = 0, & f \in E, f \ne e, \delta_{f,\vartheta} =0,\\
 \alpha(\m{p}f[\m{p}f]^*) = \alpha(ff^*), & \m{p} \in \m{P}, f \in E, f \ne e, \\
 \alpha(s(\bar g)) = \alpha(r(\bar g)), & \bar g \in E, \bar g \ne e, \delta_{\bar g, \vartheta} = 1.
\end{cases}
\end{eqnarray}

If the edge $e$ is special, \textit{i.e.,} $\alpha(ee^*) = 0$, then $\alpha(s(e)) + \alpha(r(e)) = 1$. It implies that if $\mathsf{supp}_e(s(e)) \cap \mathsf{supp}_e(r(e)) \ne \varnothing$ we then get the contradiction $1=0$.

Assume that $\mathsf{supp}_e(s(e)) \cap \mathsf{supp}_e(r(e)) = \varnothing$, we have
\begin{eqnarray*}
  \D_{ee^*} &=& \sum_{\m{p} \in \m{P}\cup V} \sum_{f \in E}\alpha(ff^*) \mathbf{ad}_{\m{p}f[\m{p}f]^*} + \sum_{v \in \widetilde{V} \setminus \{s(e),r(e)\}} \alpha(v) \mathbf{ad}_{v}\\
  &&+ \alpha(s(e))\mathbf{ad}_{s(e)} + \alpha(r(e)) \mathbf{ad}_{r(e)}.
\end{eqnarray*}

Next, (\ref{sysef}) implies that $\alpha(ff^*) = \alpha(v) - \alpha(u)$, where the edge $f \ne e$ is not special and $v \in \mathsf{supp}_e(r(f))$ and $u \in \mathsf{supp}_e(s(f))$.

We have
\begin{multline*}
  \D_{ee^*} =  \sum_{\m{p} \in \m{P}\cup V} \alpha(ee^*) \mathbf{ad}_{\m{p}e[\m{p}e]^*} + \alpha(s(e)) \mathbf{ad}_{s(e)} + \alpha(r(e)) \mathbf{ad}_{r(e)}\\
  +\sum_{\substack{ v\in \widetilde{V} \\ v \ne s(e),r(e) } }  \sum_{ \substack{  f, g \in E\setminus\{e\} \\ r(g) \in \mathsf{supp}_e(v) \\ s(f) \in \mathsf{supp}_e(v)  }  } \sum_{\m{p} \in \m{P}\cup V}  \alpha(v)\Bigl( \mathbf{ad}_{\m{p}g[\m{p}g]^* } - \mathbf{ad}_{\m{p}f[\m{p}f]^* }  +\sum_{u\in \mathsf{supp}_e(v) } \mathbf{ad}_{u} \Bigr),
\end{multline*}
where $\sum_{v\in \widetilde{V}\setminus\{s(e),r(e)\}}\sum_{u\in \mathrm{supp}_e(v) } \mathbf{ad}_{u}$ does not contain equal terms.

(1) Assume that the edge $e$ is not special. Then $\alpha(ee^*) = 1 - \alpha(s(e)) + \alpha(r(e))$ and adding up similar terms, we get
\begin{multline*}
  \D_{ee^*} =  \sum_{\m{p} \in \m{P}\cup V} \mathbf{ad}_{\m{p}e[\m{p}e]^*}\\
 +\sum_{\substack{ v\in \widetilde{V} } }  \sum_{ \substack{  f, g \in E \\ r(g) \in \mathsf{supp}_e(v) \\ s(f) \in \mathsf{supp}_e(v)  }  } \sum_{\m{p} \in \m{P}\cup V}  \alpha(v)\Bigl( \mathbf{ad}_{\m{p}g[\m{p}g]^* } - \mathbf{ad}_{\m{p}f[\m{p}f]^* }  +\sum_{u\in \mathsf{supp}_e(v) } \mathbf{ad}_{u} \Bigr).
\end{multline*}

It is easily verified that the second series has the value $0$ for every $x\in V\cup E\cup E^*$, and thus the statement follows.

(2) Let $e$ be a special edge. Then $\alpha(ee^*) = 0$ and $\alpha(s(e)) - \alpha(r(e))= 1$, and we have
\begin{eqnarray*}
  \D_{ee^*} &=& \alpha(s(e)) \Bigl( \mathsf{S}_{s(e)} +\sum_{\substack{ u \in \mathsf{supp}_e(s(e)) \\ u \ne r(e)} } \mathbf{ad}_u \Bigr) + \alpha(r(e)) \Bigl( \mathsf{S}_{r(e)} +\sum_{\substack{ u \in \mathsf{supp}_e(r(e)) \\ u \ne s(e)} } \mathbf{ad}_u \Bigr)\\
  &&+ \sum_{ \substack{ v \in \widetilde{V} \\ v \ne s(e), r(e) } } \alpha(v) \Bigl( \mathsf{S}_v + \sum_{u \in \mathsf{supp}_e(v) } \mathbf{ad}_u \Bigr),
\end{eqnarray*}
where every $\sum_{\substack{ u\in \mathrm{supp}( s(e) ) \\ u \ne r(e)} } \mathbf{ad}_u$, $\sum_{u\in \mathrm{supp}(s(e)) } \mathbf{ad}_{u}$ and $\sum_{ \substack{ v \in \widetilde{V} \\ v \ne s(e), r(e) } }\sum_{u\in \mathrm{supp}(v)}  \mathbf{ad}_{u}$  does not contain equal terms.

By the direct verification, it is easy to check that the last series has value $0$ for every $x \in V \cup E\cup E^*$, and then the statement follows.
\end{proof}

As an immediate consequence of this Lemma we get the following

\begin{corollary}\label{Dee*+}
  Let $e$ be a special edge such that $s(e) \ne r(e)$, $\mathsf{supp}_e(s(e)) \cap \mathsf{supp}_e(r(e)) = \varnothing$, $| \mathsf{supp}_e(s(e)) |, |\mathsf{supp}_e(r(e))| < \infty$. Then, in $HH^1(L_K(\Gamma))$, we have
  \[
   \D_{ee^*} = \alpha \sum_{ \substack{ f, g \in E \setminus \vartheta(V) \\ r(g) \in \mathsf{supp}_e(s(e)) \\ s(f) \in  \mathsf{supp}_e(s(e))  } } \Bigl( \D_{gg^*}  - \D_{ff^*} \Bigr) + \beta \sum_{ \substack{ f, g \in E \setminus \vartheta(V) \\ r(g) \in \mathsf{supp}_e(r(e)) \\ s(f) \in  \mathsf{supp}_e(r(e))  } } \Bigl( \D_{gg^*}  - \D_{ff^*} \Bigr),
  \]
here $\alpha,\beta \in K$, $\alpha - \beta =1$.
\end{corollary}

\begin{lemma}\label{Def}
  Let $e,f \in E$ be two edges such that $s(e) = s(f)$ and $r(e) = r(f)$. Suppose that $\D_{ef^*}(x) = \sum_{\m{wh^*} \in \m{B}} \alpha(\m{wh^*})\mathbf{ad}_{\m{wh^*}}(x)$ for every $x\in L_K(\Gamma)$, then $\D_{ef^*}(x) = \sum_{\m{p} \in \m{P}\cup V} \mathbf{ad}_{\m{p}e[\m{p}f]^*}(x)$.
\end{lemma}
\begin{proof}
  Indeed, by Corollary \ref{genofDer}, the family $\{F_e(f) = 1, F_{f^*}(e^*) = - 1\}$ determines the derivation $\D_{ef^*}$. By (\ref{I1}), (\ref{I1*}), we have
\begin{align*}
  & F_e(f) = \delta_{e,f}\alpha(s(e)) - \delta_{e,f}\alpha(r(e)) + \alpha(ef^*) = 1,\\
  & F_{f^*}(e^*) = - \delta_{e,f} \alpha(s(f)) + \delta_{e,f} \alpha(r(f)) - \alpha(ef^*) = -1,
\end{align*}
and the statement follows.
\end{proof}

\begin{lemma}\label{ee*inM}
  Let $e\in E$ be a special edge and $s(e) = r(e)$. Then, $\D_{ee^*}(x)$ cannot be presented as a series $\sum_{\m{wh^*} \in \m{B}} \alpha(\m{wh^*})\mathbf{ad}_{\m{wh^*}}(x)$ for every $x\in L_K(\Gamma)$.
\end{lemma}
\begin{proof}
  Indeed, by Corollary \ref{genofDer}, the family $\{F_e(e) = 1, F_{e^*}(e^*) = - 1\}$ determines the derivation $\D_{ee^*}$. By (\ref{I1}), $\alpha(s(e)) - \alpha(r(e)) = 1$, so that we get $0=1$, and the statement follows.
\end{proof}

As an immediate consequence of Lemmas \ref{Dwh=ad}, \ref{Def} and \ref{ee*inM} we get the following

\begin{corollary}\label{Dwh-outer}
  Let $\m{wh^* \in M}$ and $s(\m{w}) = s(\m{h})$. The derivation $\D_{\m{wh^*}}$ is inner if and only if the set $\{\m{p} \in \m{P}: r(\m{p}) = s(\m{w})\}$ is finite.
\end{corollary}

\subsection{Summary}~\\

We summarize all our results which we have obtained. To do so, let us recall some notations.

Let $\Gamma =(V,E)$ be a row-finite graph and $L_K(\Gamma)$ the Leavitt path algebra over a commutative ring $K$. The $K$-basis $\m{B}$ of $L_K(\Gamma)$ is described by Theorem \ref{GSBLev}: $\m{B} = V \cup \m{P \cup P^* \cup M^*}$.

We fix a function $\vartheta:V \setminus \{\mbox{sinks}\} \to E$ such that $s(\vartheta(v)) = v$ for an arbitrary $v \in V \setminus \{\mbox{sinks}\}$. Denote by $\vartheta(V)$ the set of all special edges of $\Gamma$. We write $\delta_{e,\vartheta} = 1$ if the edge $e$ is special and $\delta_{e,\vartheta} = 0$ otherwise.

Let $\m{P} \ni \m{p} = p_0p_1\cdots p_{z-1}p_z$ be the decomposition of the path $\m{p}$ via the edges $p_0,\ldots, p_z\in E$. Set $\m{p}/p_0: = p_1\cdots p_{z-1}p_z$. If $\m{p}=p_0$ then we set $\m{p}/p_0:=r(\m{p})$. If $\m{p} \in E$ then $[\m{p}/p_0]^*:=r(\m{p})$.

If $\m{P}\ni\m{c} = c_0c_1\cdots c_\ell$ is a directed cycle then we have the action of $\mathbb{Z} = \langle \zeta \rangle$ on $\m{c}$ by rotations: $\zeta^0 \m{c}: = \m{c}$, $\zeta^1 \m{c}: = c_1 \cdots c_\ell c_0$, $\ldots,$ $\zeta^\ell \m{c}: = c_\ell c_0 \cdots c_{\ell -1}$, $\zeta^{\ell +1 } \m{c} = \m{c}$, \textit{etc.} Set $k(\m{c}): = \mathrm{min}_{1 \le j  \le \ell +1}\{j: \zeta^j\m{c} =\m{c}\}-1$, $c_{-1} = c_\ell$, $c_{-2} = c_{\ell-1}$, \textit{etc.}

A special path $\m{\overline p}: =  \bar e_1  \cdots \bar e_n$ in $\Gamma$ is a sequence of special edges $\bar e_1, \ldots, \bar e_n \in \vartheta(V)$ such that $r(\bar e_i) = s(\bar e_{i+1})$ or $r(\bar e_i) = r(\bar e_{i+1})$ for $i=1,\ldots, n-1$. Denote by $ \m{\overline P}$ the set of all special paths of $\Gamma$. For a vertex $v \in V$ we write $v \in \m{\overline p}$ if $v = s(\bar e_i)$ or $v = r(\bar e_i)$ for some $1 \le i \le n$. If $E \ni e \ne e_i$ for every $1 \le i \le n$ we then write $e \not \subset \m{\bar p}$. Set $\mathsf{supp}_e(v):=  \{ \m{\overline p}  \in \m{\overline P}: v\in \m{\overline p}, e \not \subset \m{\overline p } \}$ for every fixed $e\in E$ and $v \in V$.

\begin{theorem}\label{genthm}
  Let $\Gamma = (V,E)$ be a row-finite graph, $L_K(\Gamma)$ the Leavitt path algebra over a commutative ring $K$. The $K$-module $HH^1(L_K(\Gamma))$ of outer derivations of $L_K(\Gamma)$ can be presented by the following sets of generators:
    \begin{itemize}
   \item[(G1)] $\bigcup_{\substack{\m{wh^*} \in \m{M}, e \in E \\ s(\m{w}) = s(\m{h})}}\{\D_{\m{wh}^*}, \D_{ee^*}\}$, where $\D_{ee^*}(e) = e$, $\D_{ee^*}(e^*) =-e^*$, $\D_{ee^*}=0$ otherwise, and $\D_{\m{wh}^*}(h_0) = \m{w}[\m{h}/h_0]^*,$  $\D_{\m{wh}^*}(w_0^*) = -[\m{w}/w_0]\m{h}^*$, $\D_{\m{wh^*}} = 0$ otherwise;
  \item[(G2)] $\bigcup_{\substack{\m{c} \in \m{P} \\ s(\m{c}) = r(\m{c})}}\{\D_\m{c}\}$, where  $\D_\m{c}(e) = \m{c}e$, for every $e\in E$,  $\D_\m{c}(c_0^*) = -\m{c}/c_0$, and $\D_\m{c} = 0$ otherwise;
  \item[(G3)] $\bigcup_{\substack{\m{c} \in \m{P} \\ s(\m{c}) = r(\m{c})}}\{\D_{\m{c}^*}\}$, where $\D_{\m{c}^*}(e^*) = -[\m{c}e]^*$, for every $e^* \in E^*$,  $\D_{\m{c}^*}(c_0) = [\m{c}/c_0]^*$, and $\D_{\m{c}^*} = 0$ otherwise;
\end{itemize}
and the following sets of relations among those generators:
\begin{itemize}
  \item[(R1)] $\D_{\m{wh^*}}= 0$ whenever $|\{\m{p \in P}: r(\m{p}) =s(\m{w}) \}| < \infty$.\\

 \item[(R2)] $\D_{\m{wh^*}} = \sum_{e \in E}\D_{e\m{w}[e\m{h}]^*}$ whenever $|\{e \in E: r(e) = s(\m{w})\}| < \infty$.\\
 \item[(R3)] If the edge $e$ is special, $s(e) \ne r(e)$, $\mathsf{supp}_e(s(e)) \cap \mathsf{supp}_e(r(e)) = \varnothing$, $| \mathsf{supp}_e(s(e)) |, |\mathsf{supp}_e(r(e))| < \infty$, then
  \[
   \D_{ee^*} = \alpha \sum_{ \substack{ f, g \in E \setminus \vartheta(V) \\ r(g) \in \mathsf{supp}_e(s(e)) \\ s(f) \in  \mathsf{supp}_e(s(e))  } } \Bigl( \D_{gg^*}  - \D_{ff^*} \Bigr) + \beta \sum_{ \substack{ f, g \in E \setminus \vartheta(V) \\ r(g) \in \mathsf{supp}_e(r(e)) \\ s(f) \in  \mathsf{supp}_e(r(e))  } } \Bigl( \D_{gg^*}  - \D_{ff^*} \Bigr),
  \]
here $\alpha,\beta \in K$, $\alpha - \beta =1$.\\
 \item[(R4)] $\D_{\m{c}}  =\D_{\zeta^{k}\m{c}} - \sum_{e \in E} \D_{\zeta^k \m{c} ee^*}$ whenever $\delta_{c_0,\vartheta} \cdots \delta_{c_{k-1},\vartheta} = 0$, $\delta_{c_k, \vartheta}=1$ and $|\{ \m{p} \in \m{P}: r(\m{p}) = s(c_0), p_z \ne c_k \}| < \infty$. Here $\m{c}: = c_0\cdots c_\ell$ is the decomposition of $\m{c}$ via the edges, $k = k(\m{c})$.
\end{itemize}
Moreover, if we put $(e^*)^* = e$ and $(e \cdot f)^* = f^*\cdot e^*$ for every $e,f\in E\cup E^*$, then we have $(\D_\m{c})^*=- \D_{\m{c}^*}$, $(\D_{\m{wh^*}})^*= - \D_{\m{hw^*}}$ and $(\D_{ee^*})^* = - \D_{ee^*}$.
\end{theorem}

\begin{proof}
  First, from Corollaries \ref{genofDer}, \ref{Dc-outer} and \ref{Dwh-outer}, it follows that $HH^1(L_K(\Gamma))$ is generated by (G1), (G2) and (G3). Further, Remark \ref{*cor} implies the last statement.

  Next, from Corollaries \ref{Dwh=sum}, \ref{Dc=Dkc}, \ref{Dwh-outer} and \ref{Dee*+} it follows the set of relations (R1) -- (R4).

  Finally, by the straightforward verification, it is easy to check that we cannot get another relations among the generators by using series $\mathsf{S}_{\m{c}}$, $\mathsf{S}_{\m{c}^*}$, $\mathsf{S}_{\m{wh^*}}$. This completes the proof.
\end{proof}

We conclude this section with describing the dimension of the vector space of outer derivations of Leavitt path algebra.

\begin{corollary} Let $K$ be a field, for the vector space $HH^1(L_K(\Gamma))$ of outer derivation of the Leavitt path algebra $L_K(\Gamma)$, we have
\[
  \mathrm{dim} HH^1(L_K(\Gamma)) = \begin{cases}
   \infty, & \mbox{if there exists a path in $\Gamma$ with infinite length,}\\
   0, & \mbox{otherwise.}
  \end{cases}
\]
\end{corollary}
\begin{proof}
  Indeed, let the graph $\Gamma$ has at least one directed cycle (say) $\m{c}$. By Theorem \ref{genthm}, $HH^1(L_K(\Gamma))$ has infinitely many generators of form $\D_{\m{c}}$, $\D_{\m{cc}}$, $\D_{\m{ccc}}$, \textit{etc.} Further, if the graph $\Gamma$ does not contain directed cycles but it has at least one path (say) $\m{p}$ such that $|\m{p}| = \infty$, then from Theorem \ref{genthm} it follows that every its edge (say) $e$ gives the non-zero derivation $\D_{ee^*}$. Finally, if the graph $\Gamma$ is finite and it does not contain directed cycles, then by Theorem \ref{genthm}, we have no non-zero derivations.
\end{proof}

\section{Derivations of the $C^*$-algebras}
In this section we prove that, every derivation of the Leavitt path algebra $L_\mathbb{C}(\Gamma)$ can be extended to the derivation of the algebra $C^*(\Gamma)$.

An algebra $\Lambda$ with a unit $1_\Lambda$ over the complex numbers $\mathbb{C}$ is called a $*$-algebra if there is a map $*:\Lambda \to \Lambda$ such that $(x+y)^* = x^*+y^*$, $(x\cdot y)^* = y^*\cdot x^*$, $1_\Lambda^* = 1_\Lambda$, $(x^*)^* = x$ and $(\gamma\cdot x)^* = \overline \gamma \cdot x^*$ for all $x,y\in \Lambda$, $\gamma\in \mathbb{C}$, where $\overline \gamma$ denotes the complex conjugate of $\gamma$.

A $C^*$-norm on a $*$-algebra $\Lambda$ is a function $|| \cdot||: \Lambda \to \mathbb{R}^+$ for which: $||a\cdot b|| \le ||a||\cdot ||b||$; $||a+b|| \le ||a|| + ||b||$; $||a\cdot a^*|| = ||a||^2 = ||a^*||^2$; $||a|| =0$ iff $a =0$; and $||\gamma \cdot  a|| = |\gamma| \cdot  ||a||$ for all $a,b\in \Lambda$ and $\gamma \in \mathbb{C}$.

\begin{definition}
  A $C^*$-algebra is a $*$-algebra $\Lambda$ endowed with a $C^*$-norm $|| \cdot ||$, for which $\Lambda$ is complete with respect to the topology induced by $|| \cdot ||$.
\end{definition}

Let us remind a description of a $C^*$-algebra, from the operator-theoretic point of view. Let $\mathcal{H}$ be a Hilbert space, and let $\mathbf{B}(\mathcal{H)}$ denote the set of all continuous linear operator on ${H}$. A $C^*$-algebra is an adjoint-closed subalgebra of $\mathbf{B}(\mathcal{H})$ which is closed with respect to the norm topology on $\mathbf{B}(\mathcal{H})$.

A {\it partial isometry} is an element $x$ in a $C^*$-algebra $\Lambda$ for which $y = x^*x$ is a self-adjoint idempotent; that is, we have in case $y^*=y$ and $y^2 = y$.

\begin{definition}
  Let $\Gamma = (V,E)$ be a finite graph. Let $C^*(\Gamma)$ denote the universal $C^*$-algebra generated by a collection of mutually orthogonal projections $ \{ {P}_v: v \in V \}$ together with partial isometries $ \{{S}_e: e \in E \}$ which satisfy the Cuntz--Krieger relations:
  \begin{itemize}
    \item[(CK1)] ${S}^*_e{S}_e = {P}_{r(e)}$ for all $e \in E$,\\
    \item[(CK2)] ${P}_v = \sum_{s(e) =v}{S}_e{S}_e^*$ for each non-sink $v \in V$.
  \end{itemize}
  The set $\bigcup_{v\in V,\,e \in E} \{{P}_v,\,{S}_e \}$ is called a Cuntz--Krieger $\Gamma$-family.
\end{definition}

It has been proved in \cite{Tom} that there exists an injective $L_\mathbb{C}(\Gamma)$ to $C^*(\Gamma)$. Further, we have the following

\begin{theorem}{\cite[Theorem 1.2]{KPR}}\label{ThKRP}
  Let $\Gamma$ be a directed graph. Then there is a $C^*$-algebra $C^*(\Gamma)$ generated by a Cuntz--Krieger $\Gamma$-family $\{p_v,s_e \}$ of non-zero elements such that, for every Cuntz--Krieger $\Gamma$-family $\{{P}_v,{S}_e \}$ of partial isometries on $\mathcal{H}$, there is a representation $T$ of $C^*(\Gamma)$ on $\mathcal{H}$ such that $T(p_v) = {P}_v$ and $T(s_e) = {S}_e$ for all $v \in V$, $e \in E$.
\end{theorem}

Using this Theorem and our description of derivations, we can prove the following
\begin{theorem}\label{Extend}
  Let $\Gamma = (V,E)$ be a row-finite directed graph. Every derivation of the Leavitt path algebra $L_\mathbb{C}(\Gamma)$ can be extended to a derivation of the algebra $C^*(\Gamma)$.
\end{theorem}
\begin{proof}
It is sufficient to prove that every outer derivation of $L_\mathcal{C}(\Gamma)$ can be extended to a derivation of $C^*(\Gamma)$.

Let $\{s_e,p_v\}$ be a Cuntz--Krieger $\Gamma$-family which generates $C^*(\Gamma)$ and let $\{S_e,P_v\}$ be a Cuntz--Krieger $\Gamma$ family of partial isometries on $\mathcal{H}$. From Theorem \ref{ThKRP} it follows that there is representation $T$ of $C^*(\Gamma)$ on $\mathcal{H}$ such that $T(p_v) = P_v$ and $T(s_e) = S_e$ for all $e\in E$, $v\in V$.

Next, the set of generators of $HH^1(L_\mathbb{C}(\Gamma))$ has been described in Theorem \ref{genthm}. For every edge $e \in E$, for every directed cycle $\m{c} = c_0c_1\ldots c_\ell \in \Gamma$ and for every $\m{wh^*}:= w_0\cdots w_n[h_0\cdots h_m]^* \in \m{M}$ with $s(\m{w}) = s(\m{h})$, define  the following partial isometries $D_{ee^*}$, $D_{\m{c}}$, $D_{\m{c}^*}$ and $D_{\m{wh^*}}$ on $\mathcal{H}$ as follows:
\begin{itemize}
  \item[(1)] $D_{\m{c}}(S_e): = {S}_{c_0}\circ \cdots \circ {S}_{c_\ell}\circ {S}_e$, $D_\m{c}({S}_{c_0}^*) = - {S}_{c_1}\circ \cdots \circ {S}_{c_\ell}$, and if $\m{c} = c_0$ then $D_\m{c}({S}_{c_0}^*)  = - P_{s(\m{c})}$, $D_{\m{c}} = 0$ otherwise.
  \item[(2)]  $D_{\m{c}^*}(S^*_e): = - [{S}_{c_0}\circ \cdots \circ {S}_{c_\ell}\circ {S}_e]^*$, $D_\m{c^*}({S}_{c_0}) = [{S}_{c_1}\circ \cdots \circ {S}_{c_\ell}]^*$, and if $\m{c} = c_0$ then $D_\m{c^*}({S}_{c_0})  = - P_{s(\m{c})}$, $D_{\m{c^*}} = 0$ otherwise.
  \item[(3)] $D_{ee^*}(S_e) = S_e$, $D_{ee^*}(S_e^*) = -S_e^*$ and $D_{ee^*} =0$ otherwise.
  \item[(4)] $D_{\m{wh^*}}(S_{h_{0}}) = S_{w_0} \circ \cdots \circ S_{w_n} \circ [ S_{h_1} \circ \cdots \circ S_{h_m}]^*$, and if $\m{h} =h_0$ then $D_{\m{wh^*}}(S_{h_{0}}) = S_{w_0} \circ \cdots \circ S_{w_n} \circ P_{r(\m{w})}$, $D_{\m{wh^*}}(S_{w_0}) = - S_{w_1} \circ \cdots \circ S_{w_n} \circ [S_{h_0} \circ \cdots \circ S_{h_m}]^*$, and if $\m{w} = w_0$ then $D_{\m{wh^*}}(S_{w_0}) = - S_{r(\m{h})} \circ [S_{h_0} \circ \cdots \circ S_{h_m}]^*$, and $\D_{\m{wh^*}} = 0$ otherwise.
\end{itemize}

Since the composition of bounded operators is a bounded operator, then it is easy to see that all ${D}_{\m{c}}$,${D}_{\m{c}^*}:=-({D}_\m{c})^*$, ${D}_\m{wh^*}$ are bounded operators. Further, it is easy to see that these are derivations for $C^*(\Gamma)$, which satisfy the corresponding equalities of Theorem \ref{genthm}.
\end{proof}

\section{The Lie Algebra Structure of Leavitt Algebra} Let $K$ be a commutative ring. We already know that the Leavitt algebra $L_K(1,\ell)$ of order $\ell$ arise as Leavitt path algebra of a graph $R_\ell$ which has only one vertex and $\ell$ edges $e_1,\ldots, e_\ell$.

We use the following notations. Let $\m{F}(E)$ be the free monoid generated by the finite set $E = \{e_1, e_2, \ldots, e_\ell\}$. We write $\m{w} = \m{w}_\lambda \m{w}' \m{w}_\rho$, for some (possibly empty) words $\m{w},\m{w}_\lambda, \m{w}_\rho$ and $\m{w}'$ of $\m{F}(E)$.

Without loss of generality, assume that the edge $e_1$ is special. Then $\m{M}$ is the set of words of the form $\m{wh^*}$, where $\m{w,h} \in \m{F}(E)$ and their last letters are either distinct or equal, but not $e_1$.

From Theorem \ref{genthm} it follows that \textit{the $K$-module $HH^1(L_K(1,\ell))$ of outer derivations (= the first Hochschild cohomology) of $L_K(1,\ell)$ can be presented by the following sets of generators:
\begin{itemize}
   \item[(1)] $\bigcup_{ \m{wh^*} \in \m{M}, e \in E }\{\D_{\m{wh}^*}, \D_{ee^*}\}$, where $\D_{ee^*}(e) = e$, $\D_{ee^*}(e^*) =-e^*$, $\D_{ee^*} = 0$ otherwise. $\D_{\m{wh}^*}(h_0) = \m{w}[\m{h}/h_0]^*,$  $\D_{\m{wh}^*}(w_0^*) = -[\m{w}/w_0]\m{h}^*$, $\D_{\m{wh^*}} = 0$ otherwise;
  \item[(2)] $\bigcup_{ \m{w} \in \m{F}(E) }\{\D_\m{w}\}$, where  $\D_\m{w}(e) = \m{w}e$, for every $e\in E$,  $\D_\m{w}(w_0^*) = -\m{w}/w_0$, and $\D_\m{w} = 0$ otherwise;
  \item[(3)] $\bigcup_{ \m{w} \in \m{F}(E)  } \{\D_{\m{w}^*}\}$, where $\D_{\m{w}^*}(e^*) = -[\m{w}e]^*$, for every $e^* \in E^*$,  $\D_{\m{w}^*}(w_0) = [\m{w}/w_0]^*$, and $\D_{\m{w}^*} = 0$ otherwise;
\end{itemize}
and we have $\D_{\m{wh^*}} = \sum_{i=1}^\ell \D_{e_i\m{w}[e_i\m{h}]^*}$ for every $\m{wh^* \in M}$. Finally, if we put $(e^*)^* = e$, for every $e \in E$ and $(a \cdot b)^* = b^*\cdot a^*$ for every $a,b\in E\cup E^*$, then $\D_{\m{w}^*} = -\D^*_\m{w}$ and $\D_{\m{hw^*}} = -\D^*_{\m{wh^*}}$.}

As well known the $K$-module $HH^1(L_K(1,\ell))$ is in fact a Lie algebra with Lie bracket given by $[\D,\D']: = \D\circ \D' - \D'\circ \D$, for every $\D,\D' \in HH^1(L_K(1,\ell))$.

Set $\D_{\m{w}e_1[\m{h}e_1]^*}:= \D_{\m{wh^*}} - \sum_{i=2}^\ell \D_{\m{w}e_i[\m{h}e_i]^*}$ for every $\m{wh^* \in M}$. Then, by the definition of the generators of $HH^1(L_K(1,\ell))$ and straightforward computations, one can get the following formulas:
\begin{align}
  & [\D_\m{a},\D_{\m{b}} ] = \sum_{\m{b} = \m{b}_\lambda \m{b}_\rho} \D_{\m{b}_\lambda \m{a} \m{b}_\rho} - \sum_{\m{a} = \m{a}_\lambda \m{a}_\rho} \D_{\m{a}_\lambda \m{b} \m{a}_\rho},\\
  & [\D_\m{a} , \D_{\m{b}^*}] = \sum_{\substack{ \m{b} = \m{b_\lambda b_\rho} \\ \m{a} = \m{b_\rho a_\rho}  }}  \D_{\m{a_\rho b_\lambda^*}} -   \sum_{\m{a} = \m{a_\lambda b a_\rho} } \D_{ \m{a_\lambda a_\rho} } - \sum_{\substack{ \m{a = a_\lambda a_\rho} \\ \m{b = a_\rho b_\rho} }} \D_{ \m{a_\lambda b_\rho^* } }, \quad | \m{a} | \ge |\m{b}|,\\
  & [\D_{\m{ab^*}} , \D_\m{cd^*}  ] =  \sum_{\m{c = c_\lambda b c_\rho } } \D_\m{ c_\lambda a c_\rho d^* } + \sum_{\substack{ \m{c = c_\lambda b_\lambda} \\ \m{ b = b_\lambda b_\rho }  } } \D_{ \m{ c_\lambda a [d b_\rho]^* } }  - \sum_{\m{d  = d_\lambda a d_\rho } } \D_{\m{ c [d_\lambda b d_\rho ]^*  }} \notag\\
 & \phantom{[\D_{\m{ab^*}} , \D_\m{cd^*}  ] } -\sum_{ \substack{ \m{a = a_\lambda a_\rho} \\ \m{ d = d_\lambda  a_\lambda } } }\D_{ \m{ ca_\rho [d_\lambda b]^*  } } - \sum_{ \m{ a = a_\lambda d a_\rho } } \D_{ \m{ a_\lambda c a_\rho b^* } } - \sum_{\substack{ \m{a  = a_\lambda d_\lambda} \\ \m{ d =  d_\lambda d_\rho  }  } } \D_{ \m{a_\lambda  c [bd_\rho]^* } } \notag \\
 &\phantom{[\D_{\m{ab^*}} , \D_\m{cd^*}  ] }+ \sum_{ \m{b = b_\lambda c b_\rho} } \D_{\m{a [b_\lambda d b_\rho]^* }} + \sum_{\substack{ \m{c  = c_\lambda  c_\rho } \\ \m{b = b_\lambda c_\lambda } }} \D_{\m{ ac_\rho [b_\lambda d]^*  }}.
\end{align}

Let us consider a partial case when $\ell =1$, then the edge $e_1 = e$ is special, and we have $e^*e = ee^* = v$. As we have seen, the corresponding Leavitt path algebra is the Laurent polynomial algebra $K[t,t^{-1}]$, it derivation algebra is the Witt algebra.

Set
$$
e^n = \begin{cases}
  e^n, & n >0,\\
  v, & n=0,\\
  (e^{n})^*, & n <0,
\end{cases}
$$
where $e^n: = \underbrace{e \cdots e }_n$. Thus, we get $\D_{i_1,\ldots i_n}^{j_1,\ldots,j_m} :=\D_{e^n(e^{m})^*} = \D_{e^{n-m}}.$

Let us denote $\D_{i_1,\ldots,i_n}$ by $\D_n$. Using the above formulas, we have for any $m,n \in \mathbb{Z}$
\[
 \Bigl[\D_n,\D_m\Bigr] = (m-n)\D_{n+m},
\]
thus we have obtained the Witt algebra.

\newpage

\section{The Derivation Algebra of the Toeplitz Algebra}
In this section we aim to give a presentation of the Lie algebra of outer derivations of the Toeplitz algebra. We already know (see Example \ref{Toeplitz}) that the Toeplitz algebra $\mathcal{T}$ can be obtained as the Leavitt path algebra $L_K(\mathrm{T})$, where $\mathrm{T}$ is the Toeplitz graph $\xymatrix{\bullet^v \ar@(dl,ul)^a \ar@{->}[r]^b& \bullet^u}$

Set $a^n: = \underbrace{a \cdots a}_n$ and $a^{m*}: = \underbrace{a^* \cdots a^*}_m$, for $n,m \ge 1$ and $a^0  = a^{*0}: =v$. According to Theorem \ref{genthm}, the $K$-module $HH^1(\mathcal{T})$ is generated by the following set of derivations
\[
  \bigcup_{n,m \ge 1} \left\{\D_{a^n}, \D_{a^{*m}}, \D_{aa^*}, \D_{bb^*}, \D_{a^na^{*m}}, \D_{a^nbb^*}, \D_{b[a^mb]^*}, \D_{a^nb[a^mb]^*} \right\},
\]
where $\D_{a^na^{*m}}$ is zero if and only if $a$ is special. Similarly, the derivations $\D_{a^nbb^*}, \D_{b[a^mb]^*}, \D_{a^nb[a^mb]^*}$ are zero if and only if $b$ is special. Further, by (R3), $\D_{bb^*} = 0$. We have:
\begin{itemize}
  \item[Case 1.] The edge $b$ is special.   We have: if $n >m$ we get $\D_{a^na^{*m}} = \D_{a^{n-m+1}a^*}$; if $n<m$ then $\D_{a^na^{*m}} = \D_{aa^{*(m-n+1)}}$; and $\D_{a^na^{*n}} = \D_{aa^*}$. Thus, the $K$-module $HH^1(\mathcal{T})$ is free and generated by $\D_{a^n}$, $\D_{a^{*m}}$, $\D_{a^{n+1}a^*}$, $\D_{aa^{*(m+1)}}$, $n,m \ge 1$, $\D_{aa^*}$.

  \item[Case 2.] The edge $a$ is special. We have: if $n>m$ then $\D_{a^nb[a^mb]^*} = \D_{a^{n-m}bb^*}$; if $n<m$ then $\D_{a^nb[a^mb]^*} = \D_{b[a^{m-n}b]^*}$; $\D_{a^nb[a^nb]^*} = \D_{bb^*}$. Thus, the $K$-module $HH^1(\mathcal{T})$ is free and generated by $\D_{a^n}$, $\D_{a^{*m}}$, $\D_{a^nbb^*}$, $\D_{b[a^nb]^*}$, $n \ge 1$ and $\D_{aa^*}$.
\end{itemize}

Now we focus on the Lie algebra of outer derivations of the Toeplitz algebra. Without loss of generality, let us assume that $b$ is special. We have
\begin{center}
\begin{tabular}{|c|c|c|c|c|}
  \hline
& $a$ & $b$ & $a^*$ & $b^*$\\
\hline
$\D_{a^n}$ & $a^{n+1}$ & $a^nb$ & $-a^{n-1}$ & $0$\\
\hline
$\D_{a^{*n}}$ & $-a^{*(n-1)}$ & $0$ & $a^{*(n+1)}$ & $[a^nb]^*$\\
\hline
$\D_{a^na^*}$ & $a^n$ & $0$ & $-a^{n-1}a^*$ & $0$\\
\hline
$\D_{aa^{*m}}$ & $aa^{*(m-1)}$ & $0$ & $-a^{*m}$ & $0$\\
\hline
$\D_{aa^*}$ & $a$ & $0$ & $-a^*$ & $0$ \\
\hline
\end{tabular}
\end{center}

\begin{theorem}
  The Lie algebra of outer derivations of the Toeplitz algebra is presented by the set of generators: $\D_{a^n}$, $\D_{a^{*m}}$, $\D_{a^{n+1}a^*}$, $\D_{aa^{*(m+1)}}$, $n,m \ge 1$, $\D_{aa^*}$; and the set of relations among those generators:
  \begin{align*}
   &\left[ \D_{aa^*},\D_{a^n} \right] = n\D_{a^n}, \qquad \left[ \D_{a^{*n}}, \D_{aa^*} \right] = n\D_{a^{*n}}, \\
   &\left[\D_{a^n},\D_{a^m}\right] = (m-n)\D_{a^{n+m-1}},\\
   &\left[  \D_{a^{*n}}, \D_{a^{*m}} \right] = (m-n)\D_{a^{*(n+m-1)}},\\
   &\left[ \D_{a^n},\D_{a^{*m}} \right] = \begin{cases} m\D_{aa^{*(m-n+1)}} - n \D_{a^{*(m-n)}}, & m \ge n, \\ -m\D_{a^{n-m+1}a^*} + n \D_{a^{n-m}}, & m \le n, \end{cases}\\
   & \left[ \D_{aa^{*n}} , \D_{aa^{*m}} \right] = (n-m) \D_{aa^{*(m+n-1)}},\\
   &\left[ \D_{a^na^*}, \D_{a^ma^*} \right] = (n-m) \D_{a^{m+n-1}a^*},\\
   &\left[  \D_{aa^{*m}} , \D_{a^na^*}  \right] = \begin{cases} (m+n-2) \D_{aa^{*(m-n+1)}}, & m\ge n, \\ (m+n-2)\D_{a^{m-n+1}a^*}, & m \le n. \end{cases}
  \end{align*}
\end{theorem}
\begin{proof}
The proof is straightforward.
\end{proof}

\section{Computations}
Throughout this section, $\Gamma$ means a directed row-finite graph $\Gamma = (V,E)$, $K$ means an associative commutative ring with unit, and $L_K(\Gamma)$ means the Leavitt path algebra of the graph $\Gamma$ over $K$.

\begin{lemma}\label{Du}
Let $\mathscr{D}:L_K(\Gamma)\to L_K(\Gamma)$ be a $K$-linear map such that the equalities $\D(v)u + v \D(u) = \delta_{v,u}\D(v)$ hold for every $v,u \in V$. Then $\D(v) = \sum_{\m{wh^*} \in \m{B}}F_{\m{wh^*}}(v) \cdot \mathbf{ad}_{\m{wh^*}}(v).$
\end{lemma}
\begin{proof} We get $\mathscr{D}(v)v = \sum_{\m{wh^*} \in \m{B}}F_{\m{wh}^*}(v)\m{wh}^*v,$ $v\mathscr{D}(v) = \sum_{\m{wh^*} \in \m{B}}F_{\m{wh}^*}(v)v\m{wh}^*.$ The equality $\D(v)v + v \D(v) = \D(v)$ implies $\sum_{\substack{s(\m{w}) = v\\s(\m{h}) =v}}F_{\m{wh^*}}(v)\m{wh^*} = 0$. Therefore $\mathscr{D}(v) = \sum_{\substack{\m{wh^*} \in \m{B} \\ s(\m{h}) \ne v}}F_{\m{wh}^*}(v) v\m{wh}^* + \sum_{\substack{\m{wh^*}\in \m{B}\\ s(\m{w}) \ne v}}F_{\m{wh}^*}(v)\m{wh}^*v.$

Next, let $v,u$ be different vertices. It follows from preceding discussion that $\mathscr{D}(v)u = \sum_{\m{wh^*} \in\m{B}}F_{\m{wh}^*}(v)v\m{wh}^*u$ and $v\mathscr{D}(u) = \sum_{\m{wh^*}\in \m{B}}F_{\m{wh}^*}(u)v\m{wh}^*u$. To conclude the proof, it remains to note that the equality $\D(v)u + v\D(u) = 0$ implies $F_{\m{wh}^*}(s(\m{w}))=-F_{\m{wh}^*}(s(\m{h}))$.
\end{proof}

\begin{lemma}\label{De}
Let $e \in E$ and $\mathscr{D}:L_K(\Gamma) \to L_K(\Gamma)$ be a $K$-linear map which satisfies Lemma \ref{Du}. If, in addition, the equalities $\mathscr{D}(s(e))e + s(e)\mathscr{D}(e) = \mathscr{D}(e)$, $e\mathscr{D}(r(e)) +\mathscr{D}(e)r(e) = \mathscr{D}(e)$ hold, then
\[
\mathscr{D}(e) =   \sum_{\m{wh^*} \in \m{B}} F_{\m{wh^*}}(s(e)) \cdot \mathbf{ad}_{\m{wh^*}}(e)+ \sum_{\substack{\m{wh^*} \in \m{B} \\s(\m{w})= s(e)\\ s(\m{h}) = r(e)}}F_{\m{wh}^*}(e)\m{wh}^*,
\]
and $\mathscr{D}(v)e + v\mathscr{D}(e) = 0$, for $V \ni v \ne s(e)$, and $e\mathscr{D}(u) +\mathscr{D}(e)u = 0$, for $V \ni u \ne r(e).$
\end{lemma}
\begin{proof}
Let $\D(e) = \D(s(e))e + s(e)\D(e) = \D(e)r(e) + e\D(r(e))$. By Lemma \ref{Du},
\begin{eqnarray*}
  \D(e) &=& \D(s(e))e + s(e)\D(e) \\
  &=& \sum_{\m{wh^*}\in \m{B}} \Bigl( F_{\m{wh^*}}(s(e)) \mathbf{ad}_{\m{wh^*}}(s(e)) \cdot e - F_{\m{wh^*}}(e) e \m{wh^*} \Bigr),\\
  e\D(r(e)) &=& \sum_{\m{wh^*} \in \m{B}} F_{\m{wh^*}}(r(e))e\cdot \mathbf{ad}_{\m{wh^*}}(r(e)),
\end{eqnarray*}
so that
\begin{eqnarray*}
  \D(e) &=& e\D(r(e)) +  \D(e)r(e) \\
  &=& \sum_{\m{wh^*} \in \m{B}} \Bigl(F_{\m{wh^*}}(s(e)) \bigl(\m{wh^*}e - s(e)\m{wh^*}e \bigr) + F_{\m{wh^*}}(e) \cdot s(e)\m{wh^*}r(e) \Bigr)\\
  && + \sum_{\m{wh^*}\in \m{B}} F_{\m{wh^*}}(e)\Bigl(e\m{wh^*}r(e) - e r(e) \m{wh^*} \Bigr).
\end{eqnarray*}

Now if we recall $F_{\m{wh}^*}(s(\m{w}))=-F_{\m{wh}^*}(s(\m{h}))$ (see the proof of Lemma \ref{Du}), we get $ \mathscr{D}(e) =   \sum_{\m{wh^*} \in \m{B}} F_{\m{wh^*}}(s(e)) \cdot \mathbf{ad}_{\m{wh^*}}(e)+ \sum_{\substack{\m{wh^*} \in \m{B} \\s(\m{w})= s(e)\\ s(\m{h}) = r(e)}}F_{\m{wh}^*}(e)\m{wh}^*.$

Let $v \ne s(e)$. We have
\begin{eqnarray*}
  \D(v)e &=& \sum_{\m{wh^*} \in \m{B}} \mathbf{ad}_{\m{wh^*}}(v) e =  \sum_{\m{wh^*} \in \m{B}} F_{\m{wh^*}}(v) \Bigl(\m{wh^*}v - v\m{wh^*} \Bigr)e\\
   &=& - \sum_{\m{wh^*} \in \m{B}}F_{\m{wh^*}}(v) v \m{wh^*}e = - \sum_{\m{wh^*} \in \m{B}} F_{\m{wh^*}}(r(\m{h}) v\m{wh^*}e,\\
   v\D(e) &=& \sum_{\m{wh^*} \in \m{B}}F_{\m{wh^*}}(s(e)) v \cdot \mathbf{ad}_{\m{wh^*}}(e) = \sum_{\m{wh^*} \in \m{B}} F_{\m{wh^*}}(s(e)) v\m{wh^*} e\\
   &=& \sum_{\m{wh^*} \in \m{B}}F_{\m{wh^*}}(r(\m{h})) v \m{wh^*}e,
\end{eqnarray*}
therefore $\D(v)e + v\D(e) = 0$. Similarly, one can prove that $\D(e)u + e\D(u) = 0$, for $V \ni u \ne r(e)$. This completes the proof.
\end{proof}

One can prove, similarly as above, that the following statement holds.

\begin{lemma}\label{De*}
  Let $e \in E$ $\mathscr{D}:L_K(\Gamma) \to L_K(\Gamma)$ be a $K$-linear map which satisfies Lemma \ref{Du}. If the map $\D$ also satisfies the following equations
  \begin{align*}
    &\mathscr{D}(e^*)s(e)+e^*\mathscr{D}(s(e)) = \mathscr{D}(e^*), \\
    &\mathscr{D}(r(e))e^*+r(e)\mathscr{D}(e^*) = \mathscr{D}(e^*),
  \end{align*}
then $ \mathscr{D}(e^*) =   \sum_{\m{wh^*} \in \m{B}} F_{\m{wh^*}}(s(e)) \cdot \mathbf{ad}_{\m{wh^*}}(e^*)+ \sum_{\substack{\m{wh^*} \in \m{B} \\s(\m{w})= r(e)\\ s(\m{h}) = s(e)}}F_{\m{wh}^*}(e^*)\m{wh}^*$, $\mathscr{D}(v)e^* + v\mathscr{D}(e^*) = 0$, for $V \ni v \ne r(e)$, and $e^*\mathscr{D}(u) +\mathscr{D}(e^*)u = 0,$ for $V \ni u \ne s(e)$.
\end{lemma}

\begin{lemma}\label{ef}
  Let $e,f\in E$ be two edges and $\D:L_K(\Gamma) \to L_K(\Gamma)$ be a $K$-linear map. Then if the linear map $\mathscr{D}$ satisfies Lemma \ref{De} and Lemma \ref{De*}, then $\mathscr{D}(e)f + e\mathscr{D}(f) = 0$, $\mathscr{D}(e^*)f^* + e^*\mathscr{D}(f^*) = 0$, $\mathscr{D}(e)f^* + e\mathscr{D}(f^*) = 0,$ whenever $ef =0$, $e^*f^* = 0$, $ef^* = 0$, respectively.
 \end{lemma}
\begin{proof}
Indeed, since $ef=0$, then $r(e) \ne s(f)$. Using the equality $e = r(e)e$ and Lemma \ref{De}, we deduce the first statement. In a similar way, one can easy to obtain the second and the third statements.
\end{proof}

Thus, one can easy see that if a $K$-linear map $\D:L_K(\Gamma) \to L_K(\Gamma)$ satisfies Lemma \ref{De} and Lemma \ref{De*} then $\D = \mathscr{J} + \widehat{\D}$, where $\mathscr{J}$ is an inner derivation of $L_K(\Gamma)$ and $\widehat{\D}$ is a $K$-linear map is defined as follows $\widehat{\D}(v) =0$, $\widehat{\D}(e) = \sum_{\substack{\m{wh^*} \in \m{B} \\ s(\m{w}) = s(e)\\ s(\m{h}) = r(e) }}F_{\m{wh^*}}(e)$ and $\widehat{\D}(e^*) = \sum_{\substack{\m{wh^*} \in \m{B}\\ s(\m{w}) = r(e)\\ s(\m{h}) = s(e) }}F_{\m{wh^*}}(e^*)$, for every $v\in V$, $e \in E$. Our next aim is to know when $\widehat{\D}$ is a derivation.

\begin{lemma}\label{e*e}
Let $\widehat{\D}:L_K(\Gamma) \to L_K(\Gamma)$ be a $K$-linear map as above. Suppose $\widehat{\D}(e^*)f + e^*\widehat{\D}(f) = \delta_{e,f}\widehat{\D}(r(e))$, for every $e,f\in E$; then
\begin{align*}
  & F_{\m{p}}(e^*) + F_{\m{p}ff^*}(e^*) + F_{e\m{p}f}(f) = 0, \\
  & F_{\m{p}^*}(e) + F_{f[\m{p}f]^*}(e) + F_{[ e\m{p}f ]^*}(f^*) = 0,  \\
  & F_{\m{w}[f\m{h}]^*}(e^*) + F_{e\m{wh^*}}(f) = 0,
\end{align*}
whenever $s(e) = s(f)$, where $\m{p} \in \m{P} \cup V$ such that $e\m{p}f \ne 0$, $\m{wh^*} \in \m{B}$ and $h_z \ne \m{w}$. The last equation in the case $\m{w} = \m{h} \in V$ holds if and only if $e =f$.
\end{lemma}
\begin{proof}
Indeed, we have
\begin{multline*}
  \widehat{\D}(e^*)f + e^* \widehat{\D}(f) = \sum_{\substack{\m{wh^*} \in \m{B} \\ s(\m{w}) = r(e) \\ s(\m{h}) = s(e)}}F_{\m{wh^*}}(e^*) \m{wh^*}f + \sum_{\substack{\m{wh^*} \in \m{B} \\ s(\m{w}) = s(f) \\ s(\m{h}) = r(f)}}F_{\m{wh^*}}(f) e^*\m{wh^*}\\
  = \delta_{s(e),s(f)}\sum_{\substack{\m{w} \in \m{P} \cup V\\s(\m{w}) = r(e)}}\Bigl( F_{\m{w}}(e^*) + F_{\m{w}ff^*}(e^*) + F_{e\m{w}f}(f) \Bigr)\m{w}f \\
   + \delta_{s(e),s(f)}\sum_{\substack{ \m{h}^*\in\m{P}^* \cup V \\ s(\m{h}) = s(e)   }}  \Bigl(  F_{[ f\m{h}e ]^*}(e^*) + F_{ e[\m{h}e]^* }(f)  + F_{ \m{h}^* } (f)   \Bigr) [e\m{h}]^*\\
  + \delta_{s(e),s(f)}\sum_{\substack{\m{wh^*} \in \m{B} \\ s(\m{w}) = r(e) \\ s(\m{h}) = r(f)}} \Bigl( F_{\m{w}[f\m{h}]^*}(e^*) + F_{e\m{wh^*}}(f) \Bigr) \m{wh^*},
\end{multline*}
and using the equalities $\mathscr{\widehat{D}}(e^*)f + e^*\mathscr{\widehat{D}}(f) = \delta_{e,f}\mathscr{\widehat{D}}(r(e))$, $\widehat{\D}(r(e)) = 0$ we complete the proof.
\end{proof}

\begin{lemma}\label{sumee*}
   Let $e_1,\ldots, e_\ell \in E$ be a finite number of edges with a common source $v$. Let $\widehat{\D}:L_K(\Gamma) \to L_K(\Gamma)$ be a $K$-linear map which satisfies Lemma \ref{e*e}. Suppose $\sum_{i=1}^\ell\widehat{\D}(e_i)e_i^* + e_i\widehat{\D}(e_i^*)= \widehat{\D}(v)$; then $F_{e_i}(e_j) + F_{e_j^*}(e_i^*)= 0$ whenever $r(e_j) =r(e_i)$, $1 \le i,j \le \ell$.
\end{lemma}
\begin{proof}
Without loss of generality, let us assume that $e_1$ is special. We have
\begin{multline*}
 \sum_{i=1}^\ell\widehat{\D}(e_i)e_i^* + e_i\widehat{\D}(e_i^*) = \sum_{\substack{\m{w} \in \m{P} \cup V \\ s(\m{w}) = s(e_1) }} F_{\m{w}e_1}(e_1)\m{w}e_1e_1^* + \sum_{\substack{\m{h}^* \in \m{P}^* \cup V \\ s(\m{h}) = s(e_1) }} F_{[\m{h}e_1]^*}(e_1^*) e_1e_1^*\m{h^*}\\
  + \sum_{\substack{\m{w} \in \m{P} \cup V \\ s(\m{w}) = s(e_1) \\ r(\m{h}) = r(e_1) \\ w_z \ne e_1}}F_{\m{w}}(e_1)\m{w}e_1^* + \sum_{\substack{\m{h}^* \in \m{P}^* \cup V \\ s(\m{h}) = s(e_1) \\ r(\m{h}) = r(e_1) \\ h_z \ne e_1}} F_{\m{h}^*}(e_1^*) e_1\m{h}^* + \sum_{\substack{\m{wh^*} \in \m{P^*} \cup \m{M} \\ s(\m{w}) = s(e_1) }} F_{\m{wh^*}}(e_1) \m{w}[e_1\m{h}]^*\\
  + \sum_{\substack{\m{wh^*} \in \m{P} \cup \m{M} \\ s(\m{h}) = s(e_1) }} F_{\m{wh^*}}(e_1^*) e_1\m{wh^*}
 + \sum_{i=2}^\ell \sum_{\substack{\m{wh^*} \in \m{B} \\ s(\m{w}) = s(e_i) }} F_{\m{wh^*}}(e_i) \m{w}[e_i\m{h}]^* + \sum_{\substack{\m{wh^*} \in \m{B} \\ s(\m{h}) = s(e_i) }} F_{\m{wh^*}}(e_i^*) e_i\m{wh^*}.
\end{multline*}

Substituting $s(e) - \sum_{k=2}^\ell e_ke_k^*$ for $ee_1^*$ and adding up similar terms, we get:
\begin{eqnarray*}
   \left.\sum_{i=1}^\ell\widehat{\D}(e_i)e_i^*+e_i\widehat{\D}(e_i^*)\right|_V &=& \Bigl(F_{e_1}(e_1) +F_{e_1^*}(e_1^*)\Bigr)v,\\
   \left.\sum_{i=1}^\ell\widehat{\D}(e_i)e_i^*+e_i\widehat{\D}(e_i^*) \right|_\mathfrak{P} &=&
   \sum_{i=1}^\ell
   \sum_{\m{w} \in \m{P}\cup V}\Bigl(F_{e_i\m{w}e_1}(e_1)+ F_{\m{w}}(e_i^*)\Bigr)e_i\m{w},\\
  \left.\sum_{i=1}^\ell \widehat{\D}(e_i)e_i^*+e_i\widehat{\D}(e_i^*)\right|_{\mathfrak{P}^*} &=&
  \sum_{i=1}^\ell\sum_{\m{h} \in \m{P} \cup V}\Bigl(F_{\m{h}^*}(e_i)+F_{[e_i\m{h}e_1]^*}(e_1^*)\Bigr) [e_i\m{p}]^*,
\end{eqnarray*}
and
\begin{multline*}
  \left.\sum_{i=1}^\ell\widehat{\D}(e_i)e_i^*+e_i\widehat{\D}(e_i^*)\right|_\mathfrak{M} = \sum_{k=2}^\ell\Bigl(-F_{e_1}(e_1) + F_{e_k}(e_k) - F_{e_1^*}(e_1^*) + F_{e_k^*}(e_k^*)\Bigr)e_ke_k^* \\
 +
  \sum_{\substack{i=1\\k=2}}^\ell\sum_{\m{w} \in \m{P}\cup V}\Bigl(-F_{e_i\m{w}e_1}(e_1) + F_{e_i\m{w}e_k}(e_k) + F_{\m{w}e_ke_k^*}(e_i^*)\Bigr)e_i\m{w}e_ke_k^*\\
 +
  \sum_{\substack{i=1\\k=2}}^\ell\sum_{\m{h} \in \m{P}\cup V}\Bigl(-F_{[e_i\m{h}e_1]^*}(e_1^*) + F_{[e_i\m{h}e_k]^*}(e_k^*)+F_{e_k[\m{h}e_k]^*}(e_i)\Bigr)e_k[ e_i\m{h}e_k]^*\\
  + \sum_{\substack{i,j=1\\i\ne j}}^\ell\Bigl(F_{e_i}(e_j) + F_{e_j^*}(e_i^*)\Bigr)e_ie_j^*+
  \sum_{i,j=1}^\ell
  \sum_{\m{wh^*} \in \m{M}}\Bigl(F_{e_i\m{wh}^*}(e_j) + F_{\m{w}[e_j\m{h}]^*}(e_i^*)\Bigr)e_i\m{w}[e_j\m{h}]^*.
\end{multline*}

Since $\widehat{\D}(v) = 0$, we see that the equality $\sum_{i=1}^\ell\widehat{\D}(e_i)e_i^* + e_i\widehat{\D}(e_i^*) = \widehat{\D}(v)$ implies:
\begin{align*}
  &F_{e_i\m{p}e_1}(e_1) - F_{e_i\m{p}e_r}(e_r) - F_{\m{p}e_re_r^*}(e_i^*)=0, \\
  &F_{[e_i\m{p}e_1]^*}(e_1^*) - F_{e_i\m{p}e_r}(e_r^*) - F_{e_r[\m{p}e_r]^*}(e_i)^*=0,\\
  & F_{e_i}(e_j)+ F_{e_j}(e_i^*) = 0, & r(e_i) = r(e_j).
\end{align*}

It is easy to see that the first two equalities can be deduced from the first two equalities of Lemma \ref{e*e}. Indeed, by Lemma \ref{e*e}, $F_{\m{p}}(e^*) + F_{\m{p}ff^*}(e^*) + F_{e\m{p}f}(f) = 0$, so that $F_{\m{p}}(e_i^*) + F_{e_i\m{p}e_1}(e_1) = 0$ and we deduce the first equality. In the similar way one can deduce the second one. This completes the proof.
\end{proof}

\end{document}